\documentclass[11pt,eqno]{article}
\usepackage[latin1]{inputenc}
\usepackage{amssymb, amsthm, amsmath}
\usepackage{amsfonts}
\usepackage{mathrsfs}
\usepackage{graphicx, color}
\usepackage{latexsym}
\usepackage{comment}
\usepackage{setspace}
\usepackage[left=2.50cm, right=2.50cm, top=3.00cm, bottom=3.00cm]{geometry}
\newtheorem{definition}{\bf Definition}[section]
\newtheorem{theorem}[definition]{\bf Theorem}

\newtheorem{proposition}[definition]{\bf Proposition}
\newtheorem{corollary}[definition]{\bf Corollary}
\newtheorem{remark}[definition]{\bf Remark}

\allowdisplaybreaks[0]

\makeatletter

\@addtoreset{equation}{section}
\makeatother
\begin{document}
%\label{page:t}
%\thispagestyle{plain}
\title{ On a comparison theorem for parabolic equations
	\\  
	      with nonlinear boundary conditions\\[3mm]}
%%
%%\dedication{Dedicated to the memory of Professor Riichi Iino}
%%
\author{{\large Kosuke Kita}\\[3mm]
	Graduate School of Advanced Science and Engineering, \\ Waseda University, 3-4-1 Okubo Shinjuku-ku, Tokyo, 169-8555, JAPAN\\[5mm]
		{\large Mitsuharu \^{O}tani}\\[3mm]
	Department of Applied Physics, School of Science and Engineering, \\ Waseda University, 3-4-1 Okubo Shinjuku-ku, Tokyo, 169-8555, JAPAN\\[5mm]}
%% \footnotemark[1]}
\date{}
%\affiliation{Major in Pure and Applied Physics, Graduate School of Advanced Science and Engineering, \\ Waseda University, 3-4-1 Okubo Shinjuku-ku, Tokyo, 169-8555, JAPAN}
%\email{kou5619@asagi.waseda.jp}
%
%\sauthor{Mitsuharu \^{O}tani \footnotemark[2]}
%\saffiliation{Department of Applied Physics, School of Science and Engineering, \\ Waseda University, 3-4-1 Okubo Shinjuku-ku, Tokyo, 169-8555, JAPAN}
%\semail{otani@waseda.jp}
%
%\tauthor{Hiroki Sakamoto}
%\taffiliation{Hitachi-GE Nuclear Energy, Ltd.\\ 3-1-1, Saiwai-cho, Hitachi-shi, Ibaraki-ken, 317-0073, JAPAN}
%\temail{hiroki.sakamoto.ec@hitachi.com}
%
%%\fauthor{Name of 4th author}
%%\faffiliation{1st line of affiliation \\ 2nd line of affiliation}
%%\femail{e-mail address of 4th author}
%%Please do not use \date 
%\footnotetext[1]{}
\footnotetext[2]{
	The first author was partially supported
	Grant-in-Aid for JSPS Fellows \# 20J11425
	and the second author was 
	partly supported by the Grant-in-Aid 
	for Scientific Research, \# 18K03382,
	the Ministry of Education, Culture, Sports, Science and Technology, Japan.}
\footnotetext[3]{e-mail : kou5619@asagi.waseda.jp}
\maketitle
\noindent
{\bf Abstract.}
%%%%%%%%%%%%%%%%%%%%%%%%%%%%%%%%%%%%%%%%%%%%%%%%%%%%%%%%%%%%%%%%%%%%%%%%%%%%%%%
%%%%%%   Abstract               %%%%%%%%%%%%%%%%%%%%%%%%%%%%%%%%%%%%%%%%%%%%%%%
%%%%%%%%%%%%%%%%%%%%%%%%%%%%%%%%%%%%%%%%%%%%%%%%%%%%%%%%%%%%%%%%%%%%%%%%%%%%%%%
  In this paper, a new type of comparison theorem is presented for some initial-boundary value problems 
    of second order nonlinear parabolic systems with nonlinear boundary conditions.
      This comparison theorem has an advantage over the classical ones, since this makes it possible 
        to compare two solutions satisfying different types of boundary conditions.  
  Some applications are given in the last section, where
    the existence of blow-up solutions is shown for some nonlinear parabolic equations 
      and systems with nonlinear boundary conditions. 

%\newpage

%%%%%%%%%%%%%%%%%%%%%%%%%%%%%%%%%%%%%%%%%%%%%%%%%%%%%%%%%%%%%%%%%%%%%%%%%%%%%%%%%
%%%%%%%%%%   Introduction               %%%%%%%%%%%%%%%%%%%%%%%%%%%%%%%%%%%%%%%%%
%%%%%%%%%%%%%%%%%%%%%%%%%%%%%%%%%%%%%%%%%%%%%%%%%%%%%%%%%%%%%%%%%%%%%%%%%%%%%%%%%
\section{Introduction}
 Mathematical models for various types of phenomena arising from physics, chemistry, 
	 biology and so on are often described as reaction diffusion equations which give  
	   typical examples of second order nonlinear parabolic equations. 
	 It is widely recognized  that comparison theorems yield very powerful tools for analyzing 
	  the second order parabolic equations, e.g., for constructing super-solutions or 
	   sub-solutions; and for examining the asymptotic behavior of solutions.
	  On the other hand, when one chooses right boundary conditions for the heat equations, 
	    it should be noted that if no artificial control of flux is given on the boundary, 
	      it is natural to consider the nonlinear boundary conditions from a physical point of view (cf. the Stefan-Boltzmann law). 
	 However, most of the existing results on comparison theorems for nonlinear diffusion 
	   equations are concerned with the standard linear boundary conditions such as 
	    Dirichlet or Neumann boundary conditions (see \cite{QS}).
	 	  Furthermore, these comparison theorems are applicable only to problems whose  
            imposed boundary conditions are of the same form.
	 There is a result on comparison theorems dealing with nonlinear boundary conditions by 
       B\'enilan and D\'iaz \cite{Be}, which also compares two solutions satisfying 
         nonlinear boundary conditions of the same form.
 	Our comparison theorem, as is described below, has an advantage that 
       it allows us to compare solutions controlled by two different ( nonlinear ) boundary conditions.

	 The main purpose of this paper is to give a comparison theorem for a rather wide class of 
	   nonlinear systems of reaction diffusion equations with nonlinear boundary conditions, 
	     i.e., the following system of equations for $U=(u^1, u^2, \cdots, u^m)$ given by  
\begin{equation*}
 {\rm (P)} \ 
\left\{
\begin{aligned}
    & \frac{\partial u^k}{\partial t} - \sum_{i,j=1}^{N}\frac{\partial}{\partial x_j} 
       \left( a_{ij}^k(t,x)\frac{\partial u^k}{\partial x_i} \right) + \beta^k(t,x,u^k)  
       - F^k(t,x,U) \ni 0, 
      && \quad (t,x) \in Q_T := (0,T)\times \Omega,
\\
    & -\sum_{i,j=1}^{N}a_{ij}^k(t,x) ~\! \nu_{j}\frac{\partial u^k}{\partial x_i} \in \gamma^k(t,x,u^k), 
      && \quad (t,x) \in \Gamma_T := (0,T)\times \partial \Omega,
\\
    & u^k(0,x) = a^k(x), 
      && \quad x\in\Omega,
\end{aligned}
\right.
\end{equation*}
  where \(\Omega\) is a general domain in \(\mathbb{R}^N\) with smooth boundary \(\partial\Omega\),
     $\nu = \nu(x) = (\nu_1, \cdots, \nu_N)$ is the unit outward vector at $x \in \partial \Omega$, 
        \(u^k: Q_T \to \mathbb{R} \ (k =1,2,\cdots,m) \) are the unknown functions.
%%%%%%%%%%%%%

As for the coefficients \(a_{ij}^k \ (k = 1,2,\cdots,m)  \), we assume  
\begin{align} 
   \exists \lambda^k \geq 0 \quad \text{such that} \quad  
    & \lambda^k |\xi|^2  \le \sum_{i,j=1}^{N} a_{ij}^k(t,x) ~\! \xi_i \xi_j 
          \qquad \forall \xi \in\mathbb{R}^N, 
              \quad \text{a.e.} \ (t,x) \in Q_T,
   \label{ell} 
\\
   &  a_{i,j}^k \in L^\infty(Q_T), \quad a_{i,j}^k|_{\Gamma_T} \in L^\infty(\Gamma_T).
  \label{cond:a:infy}
\end{align}
  We also assume that 
    \(F^k : Q_T \times \mathbb{R}^m \to 2^{\mathbb{R}^1} \ (k=1,2,\cdots,m) \) are 
      (possibly multi-valued) nonlinear mappings; 
      \(\beta^k(t,x,\cdot)\) and \(\gamma^k(t,x,\cdot)  \ (k=1,2,\cdots,m) \) are maximal monotone graphs on 
        \(\mathbb{R}^1\times\mathbb{R}^1\) \ for $a.e.\ (t,x) $. More precisely, 
          there exist lower semi-continuous convex functions \(j^k(t,x,r) : \Gamma_T \times \mathbb{R}\to (-\infty, +\infty] \) 
           and \(\eta^k(t,x,r) : Q_T \times \mathbb{R}\to (-\infty, +\infty] \) such that   
             \(\gamma^k = \partial j^k\) and \(\beta^k = \partial \eta^k\), respectively.
    Here \(\partial j^k\) and \(\partial \eta^k\) denote subdifferentials of \(j^k\) and \(\eta^k\) 
     with respect to $r \in \mathbb{R}$, respectively.	   	 
%   {\color{red}For initial data, we assume \(a^k\in L^2(\Omega)\cap L^\infty(\Omega)\). }

  The problem with this type of boundary conditions appears in models describing diffusion phenomena 
    taking into consideration some nonlinear radiation law on the boundary (see Br\'ezis\cite{B1} 
      and Barbu \cite{Ba1}) and the solvability for (P) is examined in detail under various settings 
       (see \cite{B1,Ba1,O1}).

   In this paper, we work with solutions of (P) in the following sense.
%%%%%%%%%%%%%%%%%%%%%%%%%%%%%%%%%%%%%%%%%%%%%%%%%%%%%%%%%%%%%%%%%%%%%%%%
%%%%%%    Definition 1.1    %%%%%%%%%%%%%%%%%%%%%%%%%%%%%%%%%%%%%%%%%%%%
%%%%%%%%%%%%%%%%%%%%%%%%%%%%%%%%%%%%%%%%%%%%%%%%%%%%%%%%%%%%%%%%%%%%%%%%
\begin{definition}\label{def:1}
  A function $U = (u^1,u^2,\cdots,u^m) :Q_T \to \mathbb{R}^m$ is called 
   a super-solution (resp. sub-solution)   
    of {\rm (P)} on $[0,T]$ if and only if for all $k \in \{ 1,2,\cdots,m \}$,	
\begin{equation}\label{reg:u:k}
     u^k \in C([0,T];L^2(\Omega)) \cap L^{\infty}([0,T];L^{\infty}(\Omega))  
          \cap W^{1,2}_{loc}((0,T];L^2(\Omega))   
           \cap L^2_{loc}((0,T]; H^2(\Omega)), 
\end{equation}
   and there exist sections $ f^k, b^k, g^k \in  L^2_{loc}((0,T];L^2(\Omega))$ of 
    $ F^k(t,x,U(t,x))$, $\beta^k(t,x,u^k(t,x))$, 
\\ 
     $\gamma^k(t,x,u^k(t,x))$ satisfying 
     {\rm (P)}, i.e.,     
 \begin{equation*}
   \begin{cases}
       & \frac{\partial u^k}{\partial t} - \sum_{i,j=1}^{N}\frac{\partial}{\partial x_j} 
          \left( a_{ij}^k(t,x)\frac{\partial u^k}{\partial x_i} \right) 
             + b^k(t,x) - f^k(t,x) \geq 0 \ (\text{resp.} \ \leq 0), 
\\[2mm]
      &  \quad 
       f^k(t,x,U) \in F^k(t,x,U(t,x)), \quad  b^k(t,x) \in \beta^k(t,x,u^k(t,x)), 
        \quad  \text{a.e.} \ (t,x) \in Q_T ,
 \\[3mm]
       &  -\sum_{i,j=1}^{N}a_{ij}^k(t,x) ~\! \nu_{j}\frac{\partial u^k}{\partial x_i} 
                  \leq g^k(t,x) \ (\text{resp.} \ \geq ), 
 \\[2mm]
      &  \quad g^k(t,x) \in \gamma^k(t,x,u^k(t,x))  \quad \text{a.e.} \ (t,x) \in \Gamma_T, 
  \\[2mm] 
      &   u^k(0,x) = a^k(x), \quad \text{a.e.} \ x \in \Omega.
 \end{cases}
 \end{equation*}
   If $U$ is a super- and sub-solution of {\rm (P)} on $[0,T]$ 
     with the same sections $f^k, b^k, g^k$, 
       then $U$ is called a solution of {\rm (P)} on $[0,T]$.

   We also define the maximal existence time $T_m= T_m(U)$ of a solution $U$ by 
\begin{equation*}
   T_m(U) := \sup \{ ~\! T>0 ~\! ; ~\! U \ \text{is extended to $[0,T]$ as a solution of 
   	                                               {\rm (P)} in the sense above.} \}
\end{equation*}   
\end{definition}
%%%%%
%%%%%%%%%%%%%%%%%%%%%%%%%%%%%%%%%%%%%%%%%%%%%%%%%%%%%%%%%%%%%%%%%%%%%%%%%
%%%  Remark  1.2   %%%%%%%%%%%%%%%%%%%%%%%%%%%%%%%%%%%%%%%%%%%%%%%%%%%%%%
%%%%%%%%%%%%%%%%%%%%%%%%%%%%%%%%%%%%%%%%%%%%%%%%%%%%%%%%%%%%%%%%%%%%%%%%%
\begin{remark}
	When the existence of solution is concerned, the assumption 
	  \(D(\beta^k)\cap D(\gamma^k) \neq \emptyset\) is usually required 
	   for each $k$ (see \cite{B1,Ba1}). However we do not apparently need this assumption 
	     to derive our comparison theorem, since the existence of solutions 
	      satisfying \eqref{reg:u:k} is always assumed in our setting.
\end{remark}

%%%%%%%%%%%%%%%%%%%%%%%%%%%%%%%%%%%%%%%%%%%%%%%%%%%%%%%%%%%%%%%%%%%%%%%%%
%%%%%%%%%%%%%%%%%%%%%%%%%%%%%%%%%%%%%%%%%%%%%%%%%%%%%%%%%%%%%%%%%%%%%%%%%	
%%%%%%%%%%%%%%%%%%%%%%%%%%%%%%%%%%%%%%%%%%%%%%%%%%%%%%%%%%%%%%%%%%%%%%%%%
%%%%%%%%%%%%%%%%%%%%%%%%%%%%%%%%%%%%%%%%%%%%%%%%%%%%%%%%%%%%%%%%%%%%%%%%%	

\section{Main theorem and its proof}
 
  In this section we state our comparison theorem for (P) and give a proof of it.
   The idea of proof is standard and elementary, however, this type comparison theorem 
     can cover various types of nonlinear parabolic equations including those  
       with classical linear boundary conditions.
   The applicability of this comparison theorem will be exemplified in the next section.

   Consider the following two systems of equations:

\begin{equation*}
%\label{P1}
 {\rm (P)}_1 \ 
\left\{
\begin{aligned}
     & \frac{\partial u^k}{\partial t} 
       - \sum_{i,j=1}^{N}\frac{\partial}{\partial x_j} \left( a_{ij}^k(t,x)
          \frac{\partial u^k}{\partial x_i} \right) 
            + \beta_1^k(t,x,u^k) - F_1^k(t,x,U) \ni 0, 
         &&\quad t>0,~x\in\Omega,
\\
     & -\sum_{i,j=1}^{N}a_{ij}^k(t,x)\nu_{j}\frac{\partial u^k}{\partial x_i} 
         \in \gamma_1^k(t,x,u^k), 
         &&\quad t>0,~x\in\partial\Omega,
\\
     & u^k(0,x) = a^k_1(x), 
         &&\quad x\in\Omega,
\end{aligned}
\right.
\end{equation*}
and
\begin{equation*}
%\label{P2}
  {\rm (P)}_2
\left\{
\begin{aligned}
     & \frac{\partial u^k}{\partial t} 
      - \sum_{i,j=1}^{N}\frac{\partial}{\partial x_j} \left( a_{ij}^k(t,x)
         \frac{\partial u^k}{\partial x_i} \right) 
           + \beta_2^k(t,x,u^k) -  F_2^k(t,x,U) \ni 0, 
         &&\quad t>0,~x\in\Omega,
\\
     & -\sum_{i,j=1}^{N}a_{ij}^k(x)\nu_{j}\frac{\partial u_k}{\partial x_i} 
         \in \gamma_2^k(t,x,u^k), 
         &&\quad t>0,~x\in\partial\Omega,
\\
     & u^k(0,x) = a^k_2(x), 
         &&\quad x\in\Omega,
\end{aligned}
\right.
\end{equation*}
  where for every $k \in \{1,2, \cdots,m\}$,
   \(\beta_i^k\), \(\gamma_i^k\) and \(F_i^k\) in (P)\(_i\) satisfy 
    the same conditions as those for \(\beta^k\), \(\gamma^k\) and \(F^k\) in (P). 
     Then our main theorem is stated as follows.

%%%%%%%%%%%%%%%%%%%%%%%%%%%%%%%%%%%%%%%%%%%%%%%%%%%%%%%%%%%%%%%%%%%%%%%%%%%%%%%%%%%%%%
%%%%%%%    Theorem 2.1     %%%%%%%%%%%%%%%%%%%%%%%%%%%%%%%%%%%%%%%%%%%%%%%%%%%%%%%%%%%
%%%%%%%%%%%%%%%%%%%%%%%%%%%%%%%%%%%%%%%%%%%%%%%%%%%%%%%%%%%%%%%%%%%%%%%%%%%%%%%%%%%%%%
\begin{theorem}\label{com}%\it 
	Let \(U_1 = (u_1^1, u_1^2, \cdots, u_1^m)\) be a sub-solution of 
	 {\upshape(P)\(_1\)} on $[0,T]$ and \(U_2 = ( u_2^1, u_2^2, \cdots, u_2^m)\) 
	   be a super-solution of  {\upshape(P)\(_2\)}  on $[0,T]$, 
	    and let the following assumptions {\rm (A1)-(A4)} be satisfied.
	\begin{itemize}
		\item[{\rm (A1)}] \ \(a^k_1(x) \le a^k_2(x) \)\quad a.e.  $ x \in \Omega$ \ 
		                             for all $k \in \{ 1,2,\cdots,m \}$.
%%%%%
		\item[{\rm (A2)}] \ For each $k \in \{ 1,2,\cdots,m \}$, one of the following 
		                     {\rm (i)-(ii)} holds true. 
\\[3mm]
\hspace*{-5mm} 
		  {\rm (i)} \ $\beta_1^k(t,x,\cdot) = \beta_2^k(t,x,\cdot) = \beta^k(t,x,\cdot) 
		                      \quad a.e. \ (t,x) \in Q_T$.
  \vspace{4mm}\\
%%%%%%%%%%%
\hspace*{-5mm}
		 {\rm (ii) }  
		    $
%%%%%%%%%%%%
		      \ \sup \ \{ ~\! b_2^k ~\! ; ~\! b_2^k \in \beta_2^k(t,x,r_2) ~\! \} 
		         \leq \inf \  \{ ~\!  b_1^k ~\! ; ~\! b_1^k \in \beta_1^k(t,x,r_1) ~\!\} 
		 \\[1mm] 
		      \qquad \quad \forall r_1 \in D(\beta_1^k(t,x,\cdot)), \  
		              \forall r_2 \in D(\beta_2^k(t,x,\cdot)) 
		               \quad \text{with} \ r_1> r_2 
		                \quad a.e. \ (t,x) \in Q_T.
		   $       
 \vspace{2mm}        
%%%%%%%%%%
		\item[{\rm (A3)}] \ For each $k \in \{ 1,2,\cdots,m \}$, one of the following 
		{\rm (i)-(iii)} holds true. 
 \\[3mm] 
 \hspace*{-5mm}
  	    {\rm (i)} \ $\gamma_1^k(t,x,\cdot)  = \gamma_2^k(t,x,\cdot) 
  	                                         = \gamma^k(t,x,\cdot) $ 
  	                         \quad $ a.e. \ (t,x) \in \Gamma_T$.             
  \vspace{4mm}\\ 
%%%%%%%%%%%%%%%%%%%%%%%%%%%
 \hspace*{-5mm}
  	    {\rm (ii)}  
		  $
		    \ \sup \ \{ ~\! g_2^k ~\! ; ~\! g_2^k \in \gamma_2^k(t,x,r_2) ~\! \} 
		       \leq \inf \  \{ ~\!  g_1^k ~\! ; ~\! g_1^k \in \gamma_1^k(t,x,r_1) ~\!\} 
\\[1mm] 
		        \qquad \quad \forall r_1 \in D(\gamma_1^k(t,x,\cdot)), \  
		         \forall r_2 \in D(\gamma_2^k(t,x,\cdot)) 
		          \quad \text{with} \ r_1> r_2 
		           \quad a.e. \ (t,x) \in \Gamma_T. 
		$
 \vspace{3mm}\\
%%%%%%%%%%%%%%%%%%%
\hspace*{-5mm}
      {\rm (iii) } \  $r_1^k \leq r_2^k$ \quad 
                   $\forall r_1^k \in D(\gamma_1^k(t,x,\cdot))$, \ 
                    $\forall r_2^k \in D(\gamma_2^k(t,x,\cdot))$  
                      \quad $ a.e. \ (t,x) \in \Gamma_T$.  
    \vspace{2mm}
%%%%%%%%%%%%%%%%%
		\item[{\rm (A4)}] \ For each $k \in \{ 1,2,\cdots,m \}$, the following 
		{\rm (i)} and {\rm (ii)} hold true. 
		\\[3mm]
		\hspace*{-5mm} 
		{\rm (i)} \ 
		  $ - \infty < \sup \ \{~\! z ; z \in F_1^k(t,x,U) ~\! \} 
		         \le \inf \ \{~\! z ; z \in F_2^k(t,x,U) ~\! \} < + \infty   
		         \quad a.e. \ (t,x,U) \in Q_T \times \mathbb{R}^m $.  
\\[2mm]
		\hspace*{-5mm} 
{\rm (ii)} \ 
		    $ F_1^k(t,x,\cdot) $ or $F_2^k(t,x,\cdot)$ is single-valued and 
		      satisfies the following structure condition  
\\[1mm]
   \hspace*{3mm} {\rm (SC)} with $F^k$ replaced by $F_1^k$ or $F_2^k$:
\\[2mm]
	\hspace*{-5mm} 
		   {\rm (SC)} \ $F^k(t,x,U)$ is differentiable for almost all $U \in \mathbb{R}^m$ and 
		          satisfies
 \vspace{2mm}
	\begin{equation}\label{cond:F:sc1}
	  \frac{\partial}{\partial u_j} F^k(t,x,U) \geq 0 \quad \text{for all} \ j \neq k \quad 
	                                            \text{for} \ a.e. \ (t,x,U) \in Q_T \times \mathbb{R}^m
	\end{equation}
   and for any $M>0$ there exists $L_M>0$ such that 
     \begin{equation}\label{cond:F:sc2}
      \sup ~\! \left\{~\! \Bigl|  \frac{\partial}{\partial u_j} F^k(t,x,U)  \Bigr| ~\! ; ~\!   
                    1\leq j \leq m, \ \   
                      (t,x,U)  \in Q_T \times \{ ~\! U ~\! ; ~\! |U|_{\mathbb{R}^m} \leq M ~\!\} 
                        ~\! \right\}  \leq L_M.
     \end{equation}	
	\end{itemize}
%%%%%%%%%%%%%%%
	Then, we have 
	\begin{equation}
	\label{A2.1}
	u^k_1(t,x) \le u^k_2(t,x)\quad\quad \forall k \in \{1,2,\cdots,m\}, \quad 
	                                     \forall t\in[0,T],\ \ a.e.\ x\in\Omega.
	\end{equation}

\end{theorem}
%%%%%%%%%%%%%%%%%%%%%%%%%%%%%%%%%%%%%%%%%%%%%%%%%%%%%%%%%%%%%%%%%%%%%%%%%%%%%
%\begin{remark}
%	This type nonlinear boundary conditions \(-\partial_\nu u \in \beta(u)\) on \(\partial\Omega\) contain the classical linear boundary conditions, i.e., the homogeneous Dirichlet boundary condition and the homogeneous Neumann boundary condition.
%	The advantage of this comparison theorem is that we can obtain some qualitative property of the solutions for nonlinear problems from that of the solutions for linear problems.
%	For instance, we can prove a certain solution of nonlinear heat equations with nonlinear boundary conditions blows up in finite time by comparing the homogeneous Dirichlet boundary condition with nonlinear boundary conditions (see the next section).
%\end{remark}
%%%%%%%%%%%%%%%%%%%%%%%%%%%%%%%%%%%%%%%%%%%%%%%%%%%%%%%%%%%%%%%%%%%%%%%%%%%%%%%%

\begin{proof}
	  Let $f_i^k, \ b_i^k, \ g_i^k$ be the sections of 
	    $ F_i^k(U_i), \ \beta^k(u^k_i), \ \gamma^k(u^k_i)$ 
	     appearing in (P)$_i$, then \(w^k := u^k_1 - u^k_2\) satisfies 
	\begin{equation}
	 \label{A22}
	  \left\{
	   \begin{aligned}
	     & \partial_t w^k - \sum_{i,j=1}^{N}\frac{\partial}{\partial x_j} 
	        \left( a_{ij}^k(t,x)\frac{\partial w^k}{\partial x_i} \right) 
	          + b^k_1 - b^k_2 \leq f_1^k(U_1) - f_2^k(U_2), 
	       && \quad (t,x) \in Q_T,
\\
	     & -\sum_{i,j=1}^{N}a_{ij}^k(t,x)\nu_{j}\frac{\partial w^k}{\partial x_i} 
	         \ge g^k_1 - g^k_2, 
	       &&\quad (t,x) \in Q_T,
\\
	     & w^k(0,x) = a^k_1(x) - a^k_2(x), 
	       && \quad x \in \Omega.
	\end{aligned}
	\right.
	\end{equation}
	Multiplying \eqref{A22} by \( (w^k)^+ := \max \ (w^k,0) \), we have 
	\begin{align*}
	 \int_{\Omega} \partial_t w^k ~\! (w^k)^+ dx 
	  - \int_{\Omega}\sum_{i,j=1}^{N}\frac{\partial}{\partial x_j} 
	     \left( a_{ij}^k(t,x)\frac{\partial w^k}{\partial x_i} \right) (w^k)^+ dx 
	     & + \int_{\Omega} (b_1^k - b_2^k) (w^k)^+ dx 
\\
	       & \quad \leq \int_{\Omega} ( f_1^k(U_1) - f_2^k(U_2) ) (w^k)^+ dx.
	\end{align*}
	Here we get 
	\begin{equation*}
	\int_{\Omega}\partial_t w^k ~\! (w^k)^+ dx 
	 = \int_{\{w^k\ge 0\}}\partial_t w^k ~\! w^k dx 
	   =\frac{1}{2}\frac{d}{dt}\int_{\{w^k\ge0\}} \!\!\! |w^k|^2 dx 
	     =\frac{1}{2}\frac{d}{dt}\int_{\Omega} |(w^k)^+|^2 dx,
	\end{equation*}
	and by \eqref{ell}
	\begin{align*}
	 & -\int_{\Omega}\sum_{i,j=1}^{N}\frac{\partial}{\partial x_j} 
	    \left( a_{ij}^k(t,x)\frac{\partial w^k}{\partial x_i} \right) (w^k)^+ dx 
\\	     
	  & \qquad \qquad = \int_{\Omega} \sum_{i,j=1}^{N} a_{ij}^k(t,x)\frac{\partial w^k}{\partial x_i}  
	       \frac{\partial (w^k)^+}{\partial x_j} dx 
	        - \int_{\partial\Omega} \sum_{i,j=1}^{N}a_{ij}^k(t,x)\nu_{j}
	           \frac{\partial w^k}{\partial x_i} (w^k)^+ d\sigma 
\\
	  & \qquad \qquad  \ge\int_{\{w^k\ge 0\}} \sum_{i,j=1}^{N} a_{ij}^k(t,x)\frac{\partial w^k}{\partial x_i}
	       \frac{\partial w^k}{\partial x_j} dx 
	        + \int_{\partial\Omega} ( g_1^k - g_2^k) (w^k)^+ d\sigma
\\
	  & \qquad \qquad   = \int_{\Omega} \sum_{i,j=1}^{N} a_{ij}^k(t,x)\frac{\partial (w^k)^+}{\partial x_i}
	       \frac{\partial (w^k)^+}{\partial x_j} dx 
	        + \int_{\partial\Omega} ( g_1^k - g_2^k ) (w^k)^+ d\sigma 
\\
	  & \qquad \qquad \ge \lambda^k \int_{\Omega} \sum_{j=1}^{N} \left| \frac{\partial (w^k)^+}{\partial x_j} \right|^2 dx + \int_{\partial\Omega} ( g_1^k - g_2^k ) (w^k)^+ d\sigma.
	\end{align*}
  Hence we have
\begin{equation}\label{A23}
   \begin{split}
	  \frac{1}{2}\frac{d}{dt}\|(w^k)^+(t)\|_{L^2}^2 
	  %+ \lambda \|\nabla w^+(t)\|_2^2 
	     + \!\int_{\partial\Omega}\!\! (g_1^k - g_2^k) (w^k)^+ d\sigma 
	      & + \!\int_{\Omega} \!\!(b_1^k - b_2^k) (w^k)^+ dx 
\\
	       & \qquad \leq  \! \int_{\Omega} \!\! ( f_1^k(U_1)-f_2^k(U_2)) (w^k)^+ dx. 
   \end{split}
\end{equation}
  Here we are going to show that 
 \begin{equation}\label{est:boundary}
     I_{\partial \Omega} := \int_{\partial\Omega} ( g_1^k - g_2^k ) ~\! (w^k)^+ d\sigma 
                            = \int_{\{u_1^k > u_2^k\}} ( g_1^k - g_2^k ) ~\! (u_1^k - u_2^k) d\sigma
                               \geq 0.
 \end{equation}
  In fact, if (i) of (A3) is satisfied, then \eqref{est:boundary} is derived from the monotonicity 
   of $\gamma^k$, and (iii) of (A3) implies $(w^k)^+|_{\partial \Omega} = 0$, which leads to
     $I_{\partial \Omega} = 0$. 
  As for the case where (ii) of (A3) is satisfied,  
    $u_1^k > u_2^k$ and $ g_1^k \in \gamma_1^k(u_1^k), \ g_2^k \in \gamma_2^k(u_2^k)$ 
     imply that 
	\begin{equation*}
	  ( g_1^k - g_2^k ) ~\! ( u_1^k - u_2^k) \geq 0,  
	\end{equation*}
   whence follows $ I_{\partial \Omega}\geq 0$.

	In the same way as above, from (A3)
	     we derive 
	\begin{equation}\label{est:positive:gamma}
	   \int_{\Omega} ( b_1^k - b_2^k ) ~\! (w^k)^+ dx  
 	       \ge 0.
	\end{equation}
  Here we consider the case where $F_1^k$ is singleton and satisfies (SC) 
     with $F^k$ replaced by $F_1^k$. 
   
	   Then by (i) of (A4) we obtain
	\begin{align}
	  \int_{\Omega}\! ( f_1^k(U_1)-f_2^k(U_2) ) (w^k)^+ dx 
	    & = \int_{\Omega}\! ( F_1^k(U_1)-f_2^k(U_2) ) (w^k)^+ dx  \notag
\\[2mm] 
	    & = \int_{\Omega}\! ( F_1^k(U_1) - F_1^k(U_2) ) (w^k)^+ dx 
	         + \!\int_{\Omega}\! ( F_1^k(U_2) - f_2^k(U_2) ) (w^k)^+ dx  \notag
\\[2mm]
	    & \le \int_{\Omega} \!( F_1^k(U_1)- F_1^k(U_2) ) (w^k)^+ dx. 
	       \label{est:df}
	\end{align}
   Furthermore by virtue of (SC), there exists some $ \theta \in (0,1)$ such that 
\begin{align*}
  I_F^k := \!\int_{\Omega}\! ( F_1^k(U_1)- F_1^k(U_2) ) (w^k)^+ dx 
       & = \int_{\Omega} \sum_{j=1}^{m} \frac{\partial}{\partial u_j} 
                  F_1^k(U_2 + \theta (U_1-U_2)) ~\! w^j ~\!  (w^k)^+ dx 
\\[2mm]
       & = \!\int_{\Omega} \sum_{j=1}^{m} \frac{\partial}{\partial u_j} 
                  F_1^k(U_2 + \theta (U_1-U_2)) ( (w^j)^+ \!- (w^j)^-)  (w^k)^+ dx 
   \\[2mm]
       & \leq  \int_{\Omega} \sum_{j=1}^{m} \frac{\partial}{\partial u_j} 
                 F_1^k(U_2 + \theta (U_1-U_2)) ~\! (w^j)^+  ~\! (w^k)^+ dx,  
\end{align*}
 where we used the fact that $w = w^+ - w^-, \ w^- := \max \ ( - w, 0) \geq 0$ 
    and $ \frac{\partial}{\partial u_j}F_1^k(U_2 + \theta (U_1-U_2)) ~\! (w^j)^- (w^k)^+
        \geq 0$ for $j \neq k$ and $(w^j)^- (w^k)^+ =0$ for $ j=k$. 
        
  Hence since $ U_i \in L^\infty(0,T; L^\infty(\Omega))$ implies that there exists
   $M>0$ such that 
\begin{equation*}
  \displaystyle\max_{i=1,2} \ \sup_{t \in (0,T)} |U_i(t)|_{\mathbb{R}^m}  \leq M, 
\end{equation*}
    we obtain by \eqref{cond:F:sc2}
  \begin{equation}\label{est:F:above}
   I_F^k \leq  L_M ~\! \|(w^k)^+\|_{L^2} ~\!  \sum_{j=1}^{m} \| (w^j)^+ \|_{L^2}.
  \end{equation}   
    Thus in view of \eqref{A23}, \eqref{est:boundary}, \eqref{est:positive:gamma} and 
     \eqref{est:F:above}, we finally get
	\begin{equation*}
	 \frac{1}{2}\frac{d}{dt} \sum_{k=1}^{m}\|(w^k)^+(t)\|_{L^2}^2 
	    \le L_M \Bigl( \sum_{k=1}^{m} \|(w^k)^+(t)\|_{L^2} \Bigr)^2 
	      \le L_M ~\! m ~\! \sum_{k=1}^{m}\|(w^k)^+(t)\|_{L^2}^2
	   \quad \forall t \in (0,T).
	\end{equation*}
	 Then integrating this over $(s,t)$ with $0<s<t \leq T$, we obtain by Gronwall's inequality
	\begin{equation*}
	 \sum_{k=1}^{m} \|(w^k)^+(t)\|_{L^2}^2  
	    \le \sum_{k=1}^{m} \|(w^k)^+(s)\|_{L^2}^2 ~\! e^{2 m L_M (t-s)}
                                           \quad 0<s\le t\le T.
	\end{equation*}
	Since \(w^k \in C([0,T];L^2(\Omega))\), letting \(s\to 0\), we obtain by (A1)  
	\begin{equation*}
	%\label{A24}
	     \sum_{k=1}^{m} \|(w^k)^+(t)\|_{L^2}^2
	        \le  \sum_{k=1}^{m} \|(a_1^k - a_2^k)^+\|_{L^2}^2  ~\! e^{2 m L_M T} = 0 
	     \quad\quad\quad \forall t \in [0,T],
	\end{equation*}
	whence follows \eqref{A2.1}. 
  
  As for the case where $F_2^k$ is singleton and satisfies (SC) with $F^k$ replaced by 
    $F_2^k$, instead of \eqref{est:df} we can get 
\begin{equation*}
  \int_{\Omega}\! ( f_1^k(U_1)-f_2^k(U_2) ) (w^k)^+ dx 
        \le \int_{\Omega} \!( F_2^k(U_1)- F_2^k(U_2) ) (w^k)^+ dx. 
\end{equation*} 
  Then we can repeat the same argument as above with $F_1^k$ replaced by 
   $F_2^k$.
\end{proof}
%%%%%%%%%%%%%%%%%%%%%%%%%%%%%%%%%%%%%%%%%%%%%%%%%%%%%%%%%%%%%%%%%%%%%%%%
%%%%%%%   Remark 2.2      %%%%%%%%%%%%%%%%%%%%%%%%%%%%%%%%%%%%%%%%%%%%%% 
%%%%%%%%%%%%%%%%%%%%%%%%%%%%%%%%%%%%%%%%%%%%%%%%%%%%%%%%%%%%%%%%%%%%%%%%

\begin{remark}\label{R22}
 {\rm (1)} \ If $f_1^k(U_1) \leq f_2^k(U_2)$ is known a priori, we need not assume {\rm (A4)} 
          for $F_1^k$ and $F_2^k$ in Theorem \ref{com}.
\\[2mm]
 {\rm (2)} \ If $g_1^k(u_1^k) \leq g_2^k(u_2^k)$ is known a priori, we need not assume {\rm (A3)} 
         for $\gamma_1^k$ and $\gamma_2^k$ in Theorem \ref{com}.
\\[2mm]
{\rm (3)} \ If $m=1$ in Theorem \ref{com}, then assumption \eqref{cond:F:sc1} is not needed.
\\[2mm]
{\rm (4)} \ When we discuss the existence of solutions for {\rm (P)$_i$} ($i=1,2$), 
 we need to assume that $\beta_i^k$ and $\gamma_i^k$ are maximal monotone 
   graphs. In Theorem \ref{com}, however, we need only the monotonicity of 
    $\beta_i^k$ and $\gamma_i^k$, since the existence of solutions is always assumed 
       in our setting.
\\[2mm]
{\rm (5)} \ The following condition gives a sufficient condition for {\rm (ii)} of {\rm (A3)}.
\\[3mm]
 {\rm (ii)'}  
$
   \begin{cases}		    
    \  D(\gamma_1^k(t,x,\cdot))\subset D(\gamma_2^k(t,x,\cdot)) 
     \quad a.e. \ (t,x) \in \Gamma_T, \quad \text{and} 
\\[2mm]
    \ \inf \  \{ ~\!  g_1^k ~\! ; ~\! g_1^k \in \gamma_1^k(t,x,r) ~\!\} 
       \geq \sup \ \{ ~\! g_2^k ~\! ; ~\! g_2^k \in \gamma_2^k(t,x,r) ~\! \} 
        \quad \forall r \in D(\gamma_1^k(t,x,\cdot)),
         %\quad a.e. \ (t,x) \in \Gamma_T, 
    \end{cases}
$
\\[2mm]
  and the same assertion for {\rm (ii)} of {\rm (A2)} as above holds true.
\end{remark}

%%%%%%%%%%%%%%%%%%%%%%%%%%%%%%%%%%%%%%%%%%%%%%%%%%%%%%%%%%%%%%%%%%%%%%%%
%%%%%    3  Applications                  %%%%%%%%%%%%%%%%%%%%%%%%%%%%%%
%%%%%%%%%%%%%%%%%%%%%%%%%%%%%%%%%%%%%%%%%%%%%%%%%%%%%%%%%%%%%%%%%%%%%%%%
%%%%%%%%%%%%%%%%%%%%%%%%%%%%%%%%%%%%%%%%%%%%%%%%%%%%%%%%%%%%%%%%%%%%%%%%

\section{Applications}
{
 In this section we give a couple examples of the application of 
  our comparison theorem to some nonlinear problems.
    Especially, in {\S} 3.1, we give a simple proof of the existence of 
      blowing-up solutions for nonlinear diffusion equations 
        with nonlinear boundary conditions.  
   
   We also discuss in {\S} 3.2 the finite time blow up of solutions for 
     a reaction diffusion system arising from a nuclear model 
       with nonlinear boundary conditions, which consists of two equations 
         possessing a nonlinear coupling term between two real-valued 
           unknown functions.
}

%%%%%%%%%%%%%%%%%%%%%%%%%%%%%%%%%%%%%%%%%%%%%%%%%%%%%%%%%%%%%%%%%%%%%%%%%%%%%%%%%%%%%%%
%%%%%%%  3.1.  Nonlinear heat equations with nonlinear boundary conditions     %%%%%%%%
%%%%%%%%%%%%%%%%%%%%%%%%%%%%%%%%%%%%%%%%%%%%%%%%%%%%%%%%%%%%%%%%%%%%%%%%%%%%%%%%%%%%%%%
\subsection{Nonlinear heat equations with nonlinear boundary conditions}
{
 Consider the following nonlinear heat equations with nonlinear boundary conditions:
\begin{equation}\label{A31}
{\rm (P)}^\gamma_{{\rm F}} \ 
\left\{
  \begin{aligned}
    & \ \partial_t u -\Delta u - F(u) \ni 0, 
       && \quad\quad\quad t>0,~x\in\Omega,
\\
    & \ -\partial_\nu u \in \gamma(u), 
       && \quad\quad\quad t>0, ~ x\in\partial\Omega,
\\
    & \ u(0,x) = u_0(x) \ge 0,
       && \quad\quad\quad x\in\Omega.
\end{aligned}
\right.
\end{equation}
   Here \(\Omega\) is a bounded domain in \(\mathbb{R}^N\) 
    with smooth boundary \(\partial\Omega\) and $\partial_{\nu}$ denotes 
      the outward normal derivative, i.e., $\partial_{\nu} u = \nabla u\cdot \nu$.
       We further impose the following assumptions 
         on $F$ and $\gamma$.
\begin{itemize}
%%%%%%%%%%%%%%%%%
	\item[{\rm (F)}] \ $F : \mathbb{R}^1 \to 2^{\mathbb{R}^1}$ is 
	   a (possibly multi-valued) operator satisfying the following (i) and (ii).
 \begin{align}
  {\rm (i)} & \  0 \in F(0), \quad \inf \ \{ ~\! z ~\!; ~\! z \in F(u) ~\! \} 
                        \geq |u|^{p-2} u^+ \quad \forall u \in \mathbb{R}^1
                         \quad \text{with} \ p >2, 
   \label{cond:F1}
 \\[2mm]
   {\rm (ii)} & \ F(u) = F_s(u) + F_m^+(u) - F_m^-(u) 
      \quad \forall u \in \mathbb{R}^1 
          \quad \text{and} 
    \label{cond:F2} 
 \\
   \phantom{(ii)} & \ F_s(\cdot) \ \text{ is singleton and locally Lipschitz continuous on}    \ \mathbb{R}^1,   
  \notag
 \\ 
   \phantom{(ii)} & \ F_m^{\pm}(\cdot) : \mathbb{R}^1 \to 2^{\mathbb{R}^1} \  
  	  \text{ are maximal monotone operators such that} \ D(F_m^{\pm}) = \mathbb{R}^1.  
  	    	   \phantom{\qquad \qquad ~} 
    \notag 
  \end{align}
%%%%%%%%%%%%%%%%%	
	\item[{\rm ($\gamma$)}] \ $\gamma : \mathbb{R}^1 \to 2^{\mathbb{R}^1}$   
	     is a (possibly multi-valued) maximal monotone operator
	      satisfying $ 0 \in \gamma(0)$.
\end{itemize}

  In view of assumptions $0 \in F(0)$ and $0 \in \gamma(0)$, we immediately see that 
    \eqref{A31} possesses the trivial solution $v \equiv 0$ with sections 
       $ 0=f(v) \in F(v), \ 0 = g(v) \in \gamma(v)$. 
        Let $u$ be any solution of \eqref{A31} with $ u_0(x)\geq 0$ 
         with sections $ f(u) \in F(u), \ g(u) \in \gamma(u) $ 
          satisfying the regularity required in Definition \ref{def:1}, 
            whose existence is assured in 
             Proposition \ref{LWP}, then applying Theorem \ref{com} with 
     $m=1; \ F_1 = F_2 = F; \ \beta_1 = \beta_2 = 0; \ 
        \gamma_1 = \gamma_2 =\gamma; \ a_1=0, \ a_2 = u_0; \ u_1 = v =0, \ u_2 = u$, 
             we conclude that $ u \geq 0$ as far as $u$ exists. 
   Here we use the fact that $ 0 = f(u_1)  \leq  \min \{ z ; z \in F(u) \} 
     \leq f(u_2) $ is assured a priori by \eqref{cond:F1} (see Remark \ref{R22}). 
     
   Since we are here concerned with only non-negative solutions, 
    the typical model of $F$ and $\gamma$ is given by 
      $F(u) = |u|^{p-2}u$ and $ \gamma(u) = |u|^{q-2}u$. 
   For this special case, when \(q < p\), i.e., the nonlinearity 
    inside the region is stronger than that at the boundary, 
      it might be straightforward to prove that there exist solutions of \eqref{A31} 
       which blow up in finite time by applying the same strategy 
        as that in \cite{PS3}.
   Even though, it is difficult to apply such a method to \eqref{A31} 
    for the case where \(q\geq p\), and to derive the existence of blow-up solutions 
     for this case by using the variational structure, one would need some 
       complicated classifications on parameters $(p,q)$ with 
        heavy calculations ( cf. \cite{RT}). 
   We emphasize that our method for showing the existence of 
     blow-up solutions relying on Theorem \ref{com} provides us 
      a much simpler device with wider applicability. 
    
   First we state the local existence result for \eqref{A31}.
}
%%%%%%%%%%%%%%%%%%%%%%%%%%%%%%%%%%%%%%%%%%%%%%%%%%%%%%
%%%%%%%%%%%   Proposition 3.1      %%%%%%%%%%%%%%%%%%% %%%%%%%%%%%%%%%%%%%%%%%%%%%%%%%%%%%%%%%%%%%%%%%%%%%%%%
%%%%%%%%%%%%%%%%%%%%%%%%%%%%%%%%%%%%%%%%%%%%%%%%%%%%%%
\begin{proposition}\it\label{LWP}
	Let \(u_0\in L^\infty(\Omega)\), 
	 then there exists \(T_0=T_0(\|u_0\|_{L^\infty}) > 0\) such that \eqref{A31} 
	   possesses a solution \(u\) 
	      satisfying the following regularity
	\begin{equation}\label{reg:sol} 
	  u\in C([0,T_0];L^2(\Omega))\cap L^\infty(0,T_0;L^\infty(\Omega)), \quad
	    \sqrt{t}\partial_t u, \sqrt{t}\Delta u\in L^2(0,T_0;L^2(\Omega)).
	\end{equation}
	Moreover let \(T_m=T_m(u)\) be the maximal existence time of $u$,  
	 then the following alternative holds:
	\begin{itemize} 
		\item \(T_m = +\infty\) \  or 
		\item \(T_m < +\infty\), \(\lim_{t\to T_m}\|u(t)\|_{L^\infty} = +\infty\).
	\end{itemize}
\end{proposition}
%%%%%%%%%%%%%%%%%%%%%%%%%%%%%%%%%%%%%%%%%%%%%%%%%%%%%%%%%%%%%%%%%%%%%%%%%%%%%%%%%%%%%%%%%
{
\begin{proof}
  Since $\gamma$ is assumed to be maximal monotone, there exists a lower semi-continuous 
   convex function $j : \mathbb{R}^1 \to (- \infty, + \infty] $ such that 
     $j(r) \geq 0,$ and $\partial j(u) = \gamma(u)$ ( see \cite{B0}). 
  
   Define the functional $\varphi$ on \(L^2(\Omega)\) by 
  \begin{equation*}%\label{func}
    \varphi(u) = 
      \begin{cases}
       \displaystyle\frac{1}{2} \!\!\int_{\Omega} \!\! |\nabla u|^2 dx 
         +  \frac{1}{2} \!\!\int_{\Omega} \!\! |u|^2 dx
           + \int_{\partial\Omega} \!\!\! j(u) d\sigma 
       & u \in D(\varphi) := \{u\in H^1(\Omega) ;  j(u) \in L^1(\partial\Omega)\},
  \vspace{3mm}\\
        + \infty  
       &  u \in L^2(\Omega) \setminus D(\varphi).
     \end{cases}
  \end{equation*}
   Then we can see that \(\varphi \) is a lower semi-continuous convex function on $L^2(\Omega)$
     and the subdifferential operator $\partial \varphi$ associated with \(\varphi\) 
     is given as follows (see \cite{Ba1,B0,B1}):
\begin{equation*}
  \begin{cases}
    \  \partial\varphi(u) = -\Delta u + u,  
\\[2mm]
     \  D(\partial\varphi) = \{u\in H^2(\Omega) ~\!;~\! 
                               - \partial_\nu u(x) \in \gamma(u(x))  
    \quad\mbox{a.e. on }\partial\Omega\}.
  \end{cases}
\end{equation*}
  Furthermore the following elliptic estimate for \(\partial\varphi\) holds, i.e., 
    there exist some constants \(c_1\), \(c_2>0\) such that
\begin{equation}\label{elliptic}
    \|u\|_{H^2} \le c_1\|-\Delta u + u \|_{L^2} + c_2 
      \quad \forall u \in D(\partial \varphi).
\end{equation}
  Then by putting $ B(u) := - u - F(u)$, \eqref{A31} can be reduced to 
   the following abstract evolution equation 
    in $H =L^2(\Omega)$: 
\begin{equation*}
 {\rm (CP)} \ 
  \begin{cases}
    \ \displaystyle\frac{d}{dt} u(t) + \partial \varphi(u(t)) + B(u(t)) \ni 0,  
     \quad t>0, 
 \\[2mm]
     \ u(0) = u_0.  
   \end{cases}
\end{equation*}
  In order to show the existence of time local solutions of (P)$^\gamma_F$ belonging to 
   $L^\infty(\Omega)$, we rely on ``$L^\infty$-Energy Method'' 
     developed in \cite{O1}. 
  To this end, we introduce another maximal monotone graph 
   $\beta_M(\cdot) = \partial \eta_M(\cdot)$ 
      on $\mathbb{R}^1 \times \mathbb{R}^1$ by
\begin{equation*}
  \beta_{M}(r) = 
    \begin{cases}
     \ \emptyset & \quad |r|>M, 
\\ 
       \ (-\infty,0] & \quad r = - M, 
\\
        \ 0 & \quad |r| < M,   
\\ 
        \ [0, + \infty ) & \quad  r = M, 
    \end{cases} 
 \quad \quad  
   \eta_M(r) =
    \begin{cases}
      \ 0 & \quad |r| \leq M,   
\\[2mm] 
       \  + \infty  & \quad  |r| > M, 
    \end{cases}    
\end{equation*}  
   The realizations of $\beta_M$ and $\eta_M$ in $H=L^2(\Omega)$ 
     are given by  
\begin{align*}
 & \beta_{M}(u) = \partial I_{K_M}(u) =  
    \begin{cases}
       \ \emptyset & \quad |u(x)|>M, 
\\ 
       \ (-\infty,0] & \quad u(x) = - M, 
\\
       \ 0 & \quad |u(x)| < M,   
\\ 
       \ [0, + \infty ) & \quad  u(x) = M, 
    \end{cases} 
\\[3mm] 
  & I_{K_M}(u) := 
    \begin{cases}
     \ 0  & u \in K_M := \{ ~\! u \in L^2(\Omega) ~\! ; ~\! |u(x)| \leq M 
      \quad a.e. \ x \in \Omega ~\! \}, 
  \\[2mm]
      \ + \infty & u \in L^2(\Omega) \setminus K_M.
    \end{cases}     
\end{align*}  
  Here we put 
\begin{equation*}
  \varphi_M(u) := \varphi(u) + I_{K_M}(u). 
\end{equation*}
  Then we can get
\begin{equation}\label{prop:varphiM}
  \partial \varphi_M(u) = \partial \varphi(u) + \beta_{M}(u) 
    \quad \forall u \in D(\partial \varphi_M) := D(\partial \varphi) \cap K_M.
\end{equation}
  In fact, since the Yosida approximation $(\beta_M)_\lambda(\cdot) $ of 
    $\beta_M(\cdot)$ is given by 
\begin{equation*}
  (\beta_M)_\lambda(u) = 
     \begin{cases}
       \frac{u(x) + M}{\lambda} & \quad u(x) \leq - M, \hspace{40mm} ~
   \\[1mm]
       0 & \quad |u(x)| < M,
    \\[1mm]
        \frac{u(x) - M}{\lambda} & \quad u(x) \geq M,
     \end{cases}
\end{equation*}
     we easily see 
\begin{align}
  (\partial \varphi(u), (\beta_M)_\lambda (u) )_{L^2} 
    & = \int_\Omega ( - \Delta u + u) (\beta_M)_\lambda (u)~\!  dx  \notag
\\ 
    & \geq \int_\Omega (\beta_M)_\lambda '(u) |\nabla u(x)|^2 dx 
           + \int_{\partial \Omega}  - \partial_\nu u(x)~\!
         (\beta_M)_\lambda (u(x)) ~\! d\sigma 
              \geq 0.  \label{est:angle}
\end{align}
  Here we used the fact that $ u \cdot (\beta_M)_\lambda (u) \geq 0$, 
    $(\beta_M)_\lambda '(u) \geq 0$, $ - \partial_\nu u(x) \in \gamma(u(x))$ 
     and $0 \in \gamma(0)$ implies that 
      $\gamma(u) \subset (-\infty,0]$ if $u \leq 0$ and 
         $\gamma(u) \subset [0, +\infty)$ if $u \geq 0$.
   
   Consequently \eqref{est:angle} together with Theoreme 4.4 and Proposition 2.17 in \cite{B0} assures that $ \partial \varphi + \partial I_M $ becomes 
     maximal monotone. 
      Hence since $ \partial \varphi (u) + \partial I_M (u) \subset 
        \partial \varphi_M(u)$ is obvious, we can conclude that 
           \eqref{prop:varphiM} holds true.

   Now consider the following auxiliary equation:
\begin{equation*}
   {\rm (CP)}_M \ 
     \begin{cases}
       \ \displaystyle\frac{d}{dt} u(t) + \partial \varphi_M(u(t)) + B(u(t)) \ni 0,  
        \quad t>0, 
 \\[2mm]
        \ u(0) = u_0,  
      \end{cases}
\end{equation*} 
 where we choose $M>0$ such that 
\begin{equation}
     M := \|u_0\|_{L^\infty} +2.
\end{equation}
  Then we easily see that $u_0 \in \overline{D(\partial \varphi_M)}^{L^2} = K_M$.

   Define a monotone increasing function 
      $\ell(\cdot) : [0,\infty) \to [0,\infty)$ by 
\begin{equation}\label{derf:ell}
  \ell(r) := r + \sup \ \{ ~\! |z| ~\!; ~\! z \in F(\tau), \quad 
                                         |\tau| \leq r ~\!  \}.            
\end{equation}
  Here we note that $\ell(\cdot)$ takes a finite value for any finite $r$, 
    which is assured by assumption $ D(F) = D(F_m^+)=D(F_m^-)= \mathbb{R}^1$  
     and then we obtain  
\begin{equation}\label{est:B:ell}
    \sup \ \{ ~\! |z| ~\! ; ~\! z \in B(u(x)) ~\! \} 
      \leq \ell(|u(x)|).
\end{equation}
   Hence we get 
\begin{equation}\label{est:B}
    ||| B(u) |||_{L^2} := \sup \ \{ \|z\|_{L^2} ; z \in B(u) \} 
                   \leq \ell(\|u\|_{L^\infty}) ~\! |\Omega|^{1/2} 
                   \leq  \ell(M) ~\! |\Omega|^{1/2} \quad \forall u \in D(\partial \varphi_M), 
\end{equation}
   since $u \in D(\partial \varphi_M)$ implies $ \|u\|_{L^\infty} \leq M$. 
    Now we are going to check some assumptions required in \cite{O}. 
     It is easy to see that \eqref{est:B} assures assumption (A5)
       of Theorem III and (A6) of Theorem IV in \cite{O} by taking $H=L^2(\Omega)$. 
        Furthermore the compactness assumption (A1), 
         the set $\{ u ; \varphi_M(u) \leq L \}$ is compact in $H:= L^2(\Omega)$, 
           is obviously satisfied, since $\Omega$ is bounded; and the demiclosedness  
            assumption (A2) is also assured, since the maximal monotone parts 
              $F_m^{\pm}$ are always demiclosed in $L^2(\Omega)$. 
   Thus we can apply Theorem III and Corollary IV of \cite{O} to conclude that 
     \eqref{A31} admits a solution $u$ on $[0,T]$ 
       for any $T>0$ satisfying \eqref{reg:sol} with $T_0$ replaced by $T$.  
     
  Now we are going to show that there exists $T_0>0$ such that 
\begin{equation}\label{est:u:above}
   \|u(t)\|_{L^\infty} \leq M + 1 \qquad \forall t \in [0,T_0], 
\end{equation}
   whence follows $\beta_M(u(t)) = 0$ for all $t \in [0,T_0]$, 
     which implies that $u$ turns out to be the desired solution of the original equation 
       \eqref{A31} on $[0,T_0]$.
         
  To see this, multiplying (CP)$_M$ by $|u|^{r-2}u$, we get by \eqref{est:B:ell} 
\begin{align*}
  \frac{1}{r} \frac{d}{dt} \|u(t)\|_{L^r}^r 
    \!+ (r-1) \!\int_\Omega \!\! |u|^{r-2} |\nabla u(t)|^2 dx 
     + \!\!\int_{\partial \Omega} \!\!\! g(t,x) ~\! |u|^{r-2} u(t) d\sigma 
       \leq  \ell(\|u(t)\|_{L^\infty}) \|u(t)\|_{L^r}^{r-1} |\Omega|^{1/r},
\end{align*} 
  where $ g(t,x) \in \gamma(u(t,x))$ and so $g(t,x) ~\! |u|^{r-2} u(t,x) \geq 0 $. 
 Hence 
\begin{equation*}
  \frac{d}{dt} \|u(t)\|_{L^r} \leq  \ell(\|u(t)\|_{L^\infty}) ~\! |\Omega|^{1/r}. 
\end{equation*} 
  Letting $ r \to \infty$, we obtain (see \cite{O1}) 
\begin{equation}\label{est:u:Linfty}
  \|u(t)\|_{L^\infty} \leq \|u_0\|_{L^\infty} + \int_0^t \ell(\|u(s)\|_{L^\infty}) ds.
\end{equation}
   Then Lemma 2.2 of \cite{O1} assures that if we set 
\begin{equation}\label{def:T0}
  T_0 := \frac{1}{2 \ell( \|u_0\|_{L^\infty}  + 1)}, 
\end{equation} 
  then \eqref{est:u:above} holds true.  
 
 In order to prove the alternative part, assume that $T_m < \infty$ and 
   $\liminf_{t \to T_m} \|u(t)\|_{L^\infty} =: M_0 < \infty$. 
      Then there exists a sequence $\{t_n\}_{n \in \mathbb{N}}$ such that 
\begin{equation}\label{prp:tn}
  t_n \to T_m \quad \text{as} \ n \to \infty \quad \text{and} \quad
   \|u(t_n) \|_{L^\infty} \leq M_0 + 1 \qquad  \forall n \in \mathbb{N}.
\end{equation} 
   Hence in view of \eqref{def:T0}, the definition of $T_0$, regarding 
     $u(t_n)$ as an initial data, we find that $u$ can be continued up to 
       $ t_n + \frac{1}{ 2 \ell(M_0 + 2)} $ which becomes strictly larger than 
          $T_m$ for sufficiently large $n$ such that 
             $ T_m - t_n < \frac{1}{ 4 \ell(M_0 + 2)} $. This leads to a contradiction. 
               Thus the alternative assertion is verified.
\end{proof}
%%%%%%%%%%%%%%%%%%%%%%%%%%%%%%%%%%%%%%%%%%%%%%%%%%%%%%%%%%%%%%%%%%%%%%%%%%%%%%%%%%%%%
%%%%%%%%%%%%    Remark 3.2
%%%%%%%%%%%%%%%%%%%%%%%%%%%%%%%%%%%%%%%%%%%%%%%%%%%%%%%%%%%%%%%%%%%%%%%%%%%%%%%%%%%%%
\begin{remark}\label{remark3.2}
{\rm (1)} \ One can prove that under the same assumptions in Proposition \ref{LWP}, 
       problem {\rm (P)}$^\gamma_F$ with the boundary condition replaced by the 
        homogeneous Dirichlet (resp. Neumann) boundary condition, dented by 
         {\rm (P)}$^D_F$ ( resp. {\rm (P)}$^N_F$), admits a time local solution $u$ 
            satisfying \eqref{reg:sol}, which is denoted by $u^D_F$ (resp. $u^N_F$).
             To do this, it suffices to repeat the same arguments 
              as those in the proof of Proposition \ref{LWP} with obvious modifications 
               such as $j(\cdot)\equiv 0, D(\varphi) = H^1_0(\Omega)$ 
                (resp. $D(\varphi)=H^1(\Omega)$).
\\[1mm]
{\rm (2)} \ If assumption {\rm (F)} is satisfied with $F_m^- \equiv 0$, then the solution 
       of {\rm (P)}$^\gamma_F$ (or {\rm (P)}$^D_F$, {\rm (P)}$^N_F$) given in 
         Proposition \ref{LWP} is unique.   	
\end{remark}  
%%%%%%%%%%%%%%%%%%%%%%%%%%%%%%%%%%%%%%%%%%%%%%%%%%%%%%%%%%%%%%%%%%%%%%%%%%%%%%%%%%%%%%
  
  Our result on the existence of solutions of \eqref{A31} which blow up 
    in finite time can be formulated in terms of the following eigenvalue problem:  
}
\begin{equation}\label{eigen}
\left\{
\begin{aligned}
& -\Delta ~\! \phi = \lambda ~\! \phi, && x\in\Omega,\\
& \phantom{-\Delta ~\! \ } \phi = 0, &&x\in\partial\Omega.
\end{aligned}
\right.
\end{equation}
  {  
  	Let $\lambda_1>0$ be the first eigenvalue of \eqref{eigen} and 
  	  $\phi_1$ be the associated positive eigenfunction normalized 
  	    by $\int_\Omega \phi_1(x) dx =1$. 
  
  We here consider the following fully studied problem:
\begin{equation*}%\label{A32}
{\rm (P)}^D_p \ 
\left\{
\begin{aligned}
& \partial_t u -\Delta u = |u|^{p-2}u,&&\quad\quad\quad t>0,~x\in\Omega,\\
& u = 0, &&\quad\quad\quad t>0, ~ x\in\partial\Omega,\\
& u(0,x)=u_0(x)\ge 0,&&\quad\quad\quad x\in\Omega,
\end{aligned}
\right.
\end{equation*}  
  It is well known that (P)$^D_p$ admits the unique time local solution $u_p^D$ 
   for any $u_0 \in L^\infty(\Omega)$ and $T_m(u_p^D) < \infty$ if $u_0$ 
     satisfies 
\begin{equation}\label{cond:bu:D}
  u_0 \in L^\infty(\Omega), \quad 0 \leq u_0(x) \quad a.e. \ x \in \Omega, 
    \quad \text{and} \quad 
      \int_{\Omega} u_0(x) ~\! \phi_1(x) ~\! dx > \lambda_1^{\frac{1}{p-2}},  
\end{equation}
  which is proved by the so-called Kaplan's method. 

   By comparing the solution $u$ of \eqref{A31} with $u^D_p$, we obtain the following result.  
%%%%%%%%%%%%%%%%%%%%%%%%%%%%%%%%%%%%%%%%%%%%%%%%%%%%%%%%%%%%%%%%%%%%%%%%%%%%%%%%%%%%
%%%%%%%%%%%    Proposition 3.3     %%%%%%%%%%%%%%%%%%%%%%%%%%%%%%%%%%%%%%%%%%%%%%%%%
%%%%%%%%%%%%%%%%%%%%%%%%%%%%%%%%%%%%%%%%%%%%%%%%%%%%%%%%%%%%%%%%%%%%%%%%%%%%%%%%%%%%
\begin{proposition}\label{prop3.3}\it
	Assume that \(u_0\) satisfies \eqref{cond:bu:D} 
	   and let $u^\gamma_F$ be any solution of \eqref{A31}, then 
	     \( T_m(u^\gamma_F) \leq T_m(u^D_p)<\infty\), 
	        i.e., $u^\gamma_F$ blows up in finite time.
\end{proposition}
%%%%%%%%%%%%%%%%%%%%%%%%%%%%%%%%%%%%%%%%%%%%%%%%%%%%%%%%%%%%%%%%%%%%%%%%%%%%%%%%%%%%
\begin{proof}
   We apply Theorem \ref{com} with $m=1, \ a_{i,j} = \delta_{i,j}$ and  
     $\beta_1=\beta_2 = 0, \ a_1=a_2=u_0$. 
       Then (A1) and (A2) are automatically satisfied. As for (A4), 
         we take $F_1(t,x,u) = |u|^{p-2}u$ and $F_2(t,x,u) = F(u)$, 
           then \eqref{cond:F1} assures (i) of (A4), and it is clear  
             that $F_1$ satisfies (SC), since $F_1$ is of $C^1$-class 
               with respect to $u$. 
      As for the boundary conditions, we set          
\begin{align}
  & \gamma_1(r) = \gamma^D(r):= 
  	 \begin{cases}
	  \ \mathbb{R}^1 \quad &  \ \mbox{for} \ r = 0,
\\[1mm]
	  \ \emptyset \quad    & \mbox{ for} \ r \neq 0, \phantom{spacespace}
	\end{cases} \label{BCD}
%%%%%%%%%%%%%%%%%%%%%%%	
\\[1mm]
   & \gamma_2(r) = \gamma_e(r) := \
 	  \begin{cases}
       \  \gamma(r)  \quad  &  \ \mbox{for} \ r > 0,
\\[1mm]
       \ (- \infty,0] \cup \gamma(0) \quad & \ \mbox{for} \ r = 0,
\\[1mm]
       \ \emptyset \quad & \mbox{ for} \ r < 0. \phantom{spacespace}       
    \end{cases}   \label{BCgammae}   
\end{align}

   Then we can easily see that $\gamma_2$ is monotone, i.e., 
     $(z_1-z_2) (r_1 - r_2) \geq 0$ for all $[r_1,z_1], [r_2, z_2] \in \gamma_2$. 
      In fact, this is obvious when $r_i >0$ or $r_i=0 \ (i=1,2)$. 
       Let $r_1>0$ and $r_2=0$, then $z_2 \in \gamma(0)$ or $z_2 \in (-\infty,0]$.
        If $z_2 \in \gamma(0)$, the monotonicity of $\gamma$ assures the assertion; 
         and if $z_2 \in (-\infty,0]$, then since $0 \in \gamma(0)$ implies 
          $z_1\geq 0$, we get $(z_1-z_2) (r_1 - r_2) \geq z_1 ~\! r_1 \geq 0$.

   Since $ \gamma(r) \subset \gamma_2(r)$ for all $r\geq 0$ and 
     $u^\gamma_F(t,x) \geq 0 \ a.e. \ (t,x) \in \Gamma_T$, 
      which is assured by $u^\gamma_F(t,x) \geq 0 \ a.e. \ (t,x) \in Q_T$, 
       $u^\gamma_F(t,x) $ satisfies $ - \partial_\nu u^\gamma_F(t,x) \in \gamma_2(u^\gamma_F(t,x)) \ a.e. \ (t,x) \in \Gamma_T$.

  On the other hand, $ - \partial_\nu u^D_p(t,x) \in \gamma_1(u^D_p)$ 
   implies $ u^D_p(t,x) \in D(\gamma_1) = \{0\}$ 
    and $ - \partial_\nu u^D_p(t,x) \in \mathbb{R}^1$, 
      i.e., $u^D_p(t,x)$ obeys the homogeneous Dirichlet boundary condition 
       (see \cite{B0,B1,Ba1}). 

   Thus since $D(\gamma_1)=\{0\}$ and $D(\gamma_2) \subset [0,+\infty)$, (iii) of (A2) 
     is satisfied. Consequently, applying Theorem \ref{com}, we find that      
	\begin{equation*}
	   0 \leq u^D_p(t,x) \le u^\gamma_F(t,x) \quad\quad \forall t \in [0,T) \ \ a.e.\ x  \in \Omega,
	\end{equation*}
	where \(T = \min \ ( T_m(u^\gamma_F),T_m(u^D_p) ) \),
	   whence follows 
	\begin{equation}\label{Dblow}
	\|u^D_p(t)\|_{L^\infty} \le \|u^\gamma_F(t)\|_{L^\infty} \quad\quad \forall t \in [0,T).
	\end{equation}
	 Here suppose that \(T_m(u^D_p) < T_m(u^\gamma_F) \), then it follows from \eqref{Dblow} that
	\begin{equation*}
	    \lim_{t\to T_m(u^D_p)} \|u^\gamma_F(t)\|_{L^\infty} = +\infty,
	\end{equation*}
	which contradicts the definition of $T_m(u^\gamma_F)$. 
	  Hence we conclude that $T_m(u^\gamma_F) \leq T_m(u^D_p) < + \infty$.  	
\end{proof}
%%%%%%%%%%%%%%%%%%%%%%%%%%%%%%%%%%%%%%%%%%%%%%%%%%%%%%%%%%%%%%%%%%%%%%%%%%%
%%%%%%  Corollary 3.4    %%%%%%%%%%%%%%%%%%%%%%%%%%%%%%%%%%%%%%%%%%%%%%%%%%
%%%%%%%%%%%%%%%%%%%%%%%%%%%%%%%%%%%%%%%%%%%%%%%%%%%%%%%%%%%%%%%%%%%%%%%%%%%
  As the special case where $F(u) = |u|^{p-2}u$, we get the following 
    (see (2) of Remark \ref{remark3.2}).
\begin{corollary}\label{cor3.4}\it
	Assume that \(u_0\) satisfies \eqref{cond:bu:D} 
	 and let $u^\gamma_p$ be the unique solution of \eqref{A31} 
	  with $F(u) = |u|^{p-2}u$, denoted by {\rm (P)}$^\gamma_p$,   
	   then \( T_m(u^\gamma_p) \leq T_m(u^D_p)<\infty\), 
	     i.e., $u^\gamma_p$ blows up in finite time.
\end{corollary}
%%%%%%%%%%%%%%%%%%%%%%%%%%%%%%%%%%%%%%%%%%%%%%%%%%%%%%%%%%%%%%%%%%%%%%%%%%%

    We next consider another typical classical boundary condition, namely, 
     the following problem with the homogeneous Neumann boundary condition: 
\begin{equation*}%\label{A33}
	{\rm (P)}^N_p \ \left\{
	  \begin{aligned}
	& \partial_t u -\Delta u = |u|^{p-2}u,&&\quad\quad\quad t>0,~x\in\Omega,
\\
	& \partial_\nu u = 0, &&\quad\quad\quad t>0, ~ x\in\partial\Omega,
\\
	& u(0,x)u = u_0(x)\ge 0, &&\quad\quad\quad x\in\Omega.
	\end{aligned}
	\right.
	\end{equation*}
	   Then it is also well known that (P)$^N_p$ admits the unique positive local 
	    solution $u_p^N$ for any $ 0 \leq u_0 \in L^\infty(\Omega)$ and $T_m(u_p^N) < \infty$ 
	      if $u_0$ is not identically zero in $\Omega$. 
	      
	    Let $u^N_F$ be any solution of (P)$^N_F$ (see Remark \ref{remark3.2}), 
	     and we apply Theorem \ref{com} with 
	       $m=1, \ a_{i,j} = \delta_{i,j} $ and  
	         $ \beta_1 = \beta_2 = 0$, \ $\gamma_1=\gamma_2 = \gamma^N :\equiv 0, \ a_1=a_2=u_0$. 
	          Then (A1), (A2) and (A3) are automatically satisfied. 
	     As for (A4), we take $F_1(t,x,u) = |u|^{p-2}u$ and $F_2(t,x,u) = F(u)$, 
	      then \eqref{cond:F1} assures (i) of (A4), and it is clear  
	        that $F_1$ satisfies (SC).  Then we get 
\begin{equation}\label{est:uN:pF}
   \|u^N_p(t)\|_{L^\infty} \leq \|u^N_F(t)\|_{L^\infty}  
      \quad \forall t \in [0,T) \quad \text{with} \ 
        T = \min \ (T_m(u^N_p), T_m(u^N_F)), 
\end{equation}
	    whence follows 
\begin{equation}\label{est:Tm:uN}
    T_m(u^N_F) \leq T_m(u^N_p). 
\end{equation}
%%%%%%%%%%%%%%%%%%%%%%%%%%%%%%%%%%%%%%%%%%%%%%%%%%%%%%%%%%%%%%%%%%%
     We now compare (P)$^N_p$ with (P)$^\gamma_p$, i.e., (P)$^\gamma_F$ 
       with $F(u)=|u|^{p-2}u$. Let $u^\gamma_p$ be the unique non-negative solution 
         of (P)$^\gamma_p$ ( cf. (2) of Remark \ref{remark3.2} ). 
     We apply Theorem \ref{com} with 
       $m=1, \ a_{i,j} = \delta_{i,j} $ and  
       $  \beta_1=\beta_2 = 0, \ a_1=a_2=u_0, \ F_1(u)=F_2(u)=|u|^{p-2}u$. 
    Then (A1), (A2) and (A4) are satisfied. As for (A3),  
      define $\gamma_1(\cdot)$ and $\gamma_2(\cdot)$ by 
\begin{equation*}
  \gamma_1(r) = \gamma_e(r) := \
     \begin{cases}
       \  \gamma(r)  \quad  &  \mbox{for} \ r > 0,
\\[1mm]
       \ (- \infty,0] \cup \gamma(0) \quad &  \mbox{for} \ r = 0,
\\[1mm]
       \ \emptyset \quad & \mbox{for} \ r < 0,\phantom{sp}        
    \end{cases}  
%%%%%%%%%%%%%%%%%%%%%%
    \gamma_2(r) = \gamma^N_e(r) := \
      \begin{cases}
        \  0  \quad  &   \mbox{for} \ r > 0,
\\[1mm]
        \ (- \infty,0]  \quad &  \mbox{for} \ r = 0,
\\[1mm]
        \ \emptyset \quad & \mbox{for} \ r < 0.        
      \end{cases}     
\end{equation*}  
  Then we can show that $\gamma_1, \gamma_2$ are monotone by the same reasoning 
    as that for \eqref{BCgammae}. 

  Moreover since $\gamma(r) \subset \gamma_1(r)$ and \ $ 0 \equiv \gamma^N(r) \subset \gamma_2(r)$ 
   for $r\geq 0$, and $u^\gamma_p(t,x), u^N_p(t,x) \geq 0 \ a.e. \ (t,x) \in \Gamma_T$  
    are assured by $u^\gamma_p(t,x), u^N(t,x) \geq 0 \ a.e. \ (t,x) \in Q_T$,  
      we get $ - \partial_\nu u^\gamma_p(t,x) \in \gamma_1(u^\gamma_p(t,x))$ and  
        $ - \partial_\nu u^N_p(t,x) \in \gamma_2(u^N_p(t,x)) $   
          for $ a.e. \ (t,x) \in \Gamma_T$. 
          
   Furthermore for any $ r_1 \in D(\gamma_1), \ r_2 \in D(\gamma_2)$ with $r_2 < r_1$, 
    since $D(\gamma_2) = [0, + \infty)$ and $r_2 < r_1$ implies $0<r_1$   
      and $ 0 \in \gamma(0)$ is assumed, we have      
\begin{equation*}
       \sup \ \{ ~\! g_2 ~\! ; ~\! g_2 \in \gamma_2(r_2) ~\! \}  \leq     
         0 \leq \inf \  \{ ~\!  g_1 ~\! ; ~\! g_1 \in \gamma_1(r_1) ~\!\}.
\end{equation*}
   Hence (ii) of (A3) is satisfied. Consequently, applying Theorem \ref{com}, 
       we find that      
\begin{equation*}
   0 \leq u^\gamma_p(t,x) \le u^N_p(t,x)  
     \quad\quad \forall t \in [0,T) \ \ a.e.\ x  \in \Omega,
\end{equation*}
where \(T = \min \ ( T_m(u^\gamma_p), T_m(u^N_p) ) \),
whence follows 
\begin{equation}\label{Nblow} 
   T_m(u^N_p) \leq T_m(u^\gamma_p) \quad \text{and} \quad 
    \|u^\gamma_p(t)\|_{L^\infty} \le \|u^N_p(t)\|_{L^\infty} \quad
      \forall t \in [0,T_m(u^N_p)).
\end{equation}
  Thus putting arguments above all together, we obtain the following  
   observations.
%%%%%%%%%%%%%%%%%%%%%%%%%%%%%%%%%%%%%%%%%%%%%%%%%%%%%%%%%%%%%%%%%%%
%%%%%%%   Proposition 3.5     %%%%%%%%%%%%%%%%%%%%%%%%%%%%%%%%%%%%%
%%%%%%%%%%%%%%%%%%%%%%%%%%%%%%%%%%%%%%%%%%%%%%%%%%%%%%%%%%%%%%%%%%%
\begin{proposition}
 Let $u_F^\ast$ be any solution of {\rm (P)}$^\ast_F$ and 
  let $u_p^\ast$ be the unique solution of {\rm (P)}$_p^\ast$
   ($\ast = D, \gamma, N$). Then the following hold.
\begin{itemize}
	\item[(i)] \ $ T_m(u_F^D) \leq T_m(u_p^D), \  
	                  T_m(u_F^\gamma) \leq T_m(u_p^\gamma), \
	                     T_m(u_F^N) \leq T_m(u_p^N)$ .
%%%%%%%%%%%%%%%%
     \item[(ii)] \  $T_m(u^N_p) \leq T_m(u^\gamma_p) \leq T_m(u^D_p)$.
\end{itemize}	
\end{proposition}	
	   
}

%%%%%%%%%%%%%%%%%%%%%%%%%%%%%%%%%%%%%%%%%%%%%%%%%%%%%%%%%%%%%%%%%%%%%%%%%%%%%%
%%%%%%%%%    Section 3.2     %%%%%%%%%%%%%%%%%%%%%%%%%%%%%%%%%%%%%%%%%%%%%%%%%
%%%%%%%%%%%%%%%%%%%%%%%%%%%%%%%%%%%%%%%%%%%%%%%%%%%%%%%%%%%%%%%%%%%%%%%%%%%%%%
\subsection{Reaction diffusion system arising from nuclear reactor}
{
  In this subsection, we exemplify the applicability of Theorem \ref{com} 
    for systems of parabolic equations.  
     We consider the following reaction diffusion system, which consists of two   
      equations possessing a nonlinear coupling term between two 
       real-valued unknown functions.
  }
\begin{equation*}
%\tag{NR}
\label{NR}
  {\rm (NR)} \ 
    \left\{
     \begin{aligned}
      & \ \partial_t u_1-\Delta u_1 = u_1u_2 - bu_1, 
        && t>0,~x\in\Omega,
\\
      & \ \partial_t u_2-\Delta u_2 = au_1,  
        && t>0,~x\in\Omega,
\\
      & 
      \ \partial_{\nu}u_1 + \alpha_1 |u_1|^{\gamma_1 -2} u_1 
         = \partial_{\nu}u_2 + \alpha_2 |u_2|^{\gamma_2-2} u_2 = 0,
         && t>0,~x\in\partial\Omega,
\\
      & \ u_1(0,x)=u_{10}(x)\ge 0,~u_2(0,x)=u_{20}(x)\ge 0, 
         &&x\in\Omega.
\end{aligned}
\right.
\end{equation*}
Here \(\Omega\subset\mathbb{R}^N\) is a bounded domain with smooth boundary \(\partial\Omega\).
%\(\nu\) denotes the unit outward normal vector on \(\partial\Omega\) and \(\partial_\nu\) is outward normal derivative, i.e., \(\partial_\nu u_i = \nabla u_i\cdot \nu \) (\(i=1,2\)).
Moreover \(u_1\), \(u_2\) are real-valued unknown functions, \(a\) and \(b\) are given positive constants. 
  As for the parameters appearing in the boundary condition, we assume
{
   \(\alpha_i \in [0,\infty), \ \gamma_i \in (1,\infty) \ (i=1,2)\).  
     We note that the boundary condition for \(u_i\) becomes the homogeneous Neumann 
      boundary condition when \(\alpha_i = 0\), and the Robin boundary condition when 
        \( \alpha_i >0\) and \(\gamma_i = 2\).
 } 
  We further assume that the given initial data \(u_{10}\), \(u_{20}\) 
	are nonnegative and belong to \( L^{\infty}(\Omega)\).

  The equations of this system with linear boundary conditions was proposed
    in \cite{KC1} to describe the diffusion phenomenon of neutron and heat 
     in nuclear reactors, where  \(u_1\) and \(u_2\) represent the neutron density   
      and the temperature, respectively. 
%( Japanese Comment (see tex file) And, But, ???????????????? )
{
   However we here consider this system with nonlinear boundary conditions of
    power type as above, since from a physical point of view, it seems to be 
      more natural to consider the nonlinear boundary condition rather than 
       the linear ones. 
   In fact, the linear boundary conditions such as Dirichlet or Neumann type can be   
     realized only when some artificial controls of the flux are given on the boundary. 
      For a large scale system such as nuclear reactors, however, it is extremely 
        difficult to give such a control,  
          so actually in reactors no control is given for the flux on the boundary. 
   
   When there is no artificial control of the flux on the boundary, 
     there exists a well-know radiation model in physics, called      
       the Stefan-Boltzmann law, which says that the total radiant heat power 
         emitted from the boundary is proportional to the fourth power of the 
          temperature, which is far from linear.
     
  The existence and uniqueness of non-negative local solutions of (NR) belonging to 
   $L^\infty(\Omega)$ is shown in \cite{KO} for the case where $\gamma_1=2$, 
    where it is also proved that (NR) possesses a positive stationary solution 
     $ \bar{U} = (\bar{u}_1, \bar{u}_2)$ which works as the threshold to separate  
       global existence and finite time blow up for the case where 
        \(\gamma_1=\gamma_2=2\), i.e., 
         roughly speaking, if the initial data stay below $\bar{U}$, 
          then the corresponding solution exists globally, and if the initial 
            data is larger than $\bar{U}$, then the corresponding solution 
              blows up in finite time.  
   As for the case where \(\gamma_i \neq 2\), however, this method for showing the 
     existence of blow-up solutions does not work well.  
   
   Nevertheless it is possible to show that (NR) with \(\gamma_i \neq 2\) 
     admits blow-up solutions by applying the same strategy as that 
       in the previous subsection.
           Along the same lines as before, we first consider the following 
            Dirichlet problem for (NR).
\begin{equation*}%\label{DNR}
{\rm (NR)}^D \
\left\{
\begin{aligned}
      & \ \partial_t u_1-\Delta u_1 = u_1 u_2 - b u_1, 
         && t>0,~ x \in \Omega,
\\
      & \ \partial_t u_2-\Delta u_2 = a u_1,  
         && t>0,~x\in\Omega,
\\
      & \ u_1 = u_2 = 0, 
         && t>0,~x \in \partial\Omega,
\\
      & \ u_1(0,x) = u_{10}(x) \ge 0,~u_2(0,x) = u_{20}(x) \ge 0, 
         && \phantom{t>0,~}  x \in \Omega.
\end{aligned}
\right.
\end{equation*}
   We first note that for every \( U_0 := (u_{10}, u_{20} ) \in  
    \mathbb{L}^\infty_+ (\Omega) := \{ ~\! (u_1,u_2) ~\! ; ~\! u_i \geq 0, 
       u_i \in L^\infty(\Omega)  \ (i=1,2) ~\! \}\), 
         (NR) or (NR)$^D$ possess a unique solution \( U(t) := (u_1(t),u_2(t)) \in 
          \mathbb{L}^\infty_+ (\Omega) \) 
           satisfying the blow-up alternative with respect to \(L^\infty\)-norm 
            such as in Proposition \ref{LWP}. 
     We are going to show this result for a more general equation: 
 \begin{equation*}%\label{DNR}
 {\rm (NR)}^\gamma \
 \left\{
 \begin{aligned}
 & \ \partial_t u_1-\Delta u_1 = u_1 u_2 - b u_1, 
 && t>0,~ x \in \Omega,
 \\
 & \ \partial_t u_2-\Delta u_2 = a u_1,  
 && t>0,~x\in\Omega,
 \\
 & \   \partial_{\nu}u_1 + \gamma_1(u_1)  
          = \partial_{\nu}u_2 + \gamma_2 (u_2) = 0,
 && t>0,~x \in \partial\Omega,
 \\
 & \ u_1(0,x) = u_{10}(x) \ge 0,~u_2(0,x) = u_{20}(x) \ge 0, 
 && \phantom{t>0,~}  x \in \Omega,
 \end{aligned}
 \right.
 \end{equation*}    
  where $\gamma_i : \mathbb{R}^1 \to 2^{\mathbb{R}^1}$ are 
   maximal monotone operators (\(i=1,2\)).  
    To do this, we can repeat much the same arguments as those in the proof 
     of Proposition \ref{LWP}.
     
   Let $H := L^2(\Omega) \times L^2(\Omega)$ with inner product 
     $(U,V)_H := (u_1,v_1)_{L^2} + (u_2, v_2)_{L^2}$ for $U=(u_1,u_2),\  V=(v_1,v_2)$, 
       and put $|\nabla U|^2 = |\nabla u_1|^2 + |\nabla u_2|^2$.
   Let $j_i : \mathbb{R}^1 \to (- \infty, + \infty]$ be lower semi-continuous convex  
     functions such that $ \partial j_i = \gamma_i$ \((i=1,2)\). 
      For the Dirichlet (resp. Neumann) boundary condition, we put 
       $j_i(0)=0$ and $j_i(r)= + \infty$ for $r \neq 0$ ( resp. $j_i(r)=0, \ \forall r \in \mathbb{R}^1$ ). 
       
       Then we define     
\begin{equation*}
 \varphi(U) = 
              \begin{cases} 
                 \ \displaystyle\frac12 \int_\Omega ( |\nabla U(x)|^2 + |U(x)|^2) dx 
                    + \sum_{i=1}^{2} \int_{\partial \Omega} j_i(u_i(x)) d\sigma 
         & U \in D(\varphi), 
        \\[2mm] 
                \ + \infty & U \in H \backslash D(\varphi), 
              \end{cases}
\end{equation*}
  where $ D(\varphi):= \{ U ; u_i \in H^1(\Omega) \ j_i(u_i) \in L^1(\Omega) \ (i=1,2) \}$. 
    For the homogeneous Dirichlet (resp. Neumann) boundary condition case, we take 
     $D(\varphi)= H^1_0(\Omega) \times  H^1_0(\Omega)$ \  
      (resp. $H^1(\Omega) \times H^1(\Omega)$). 
   Then we have 
    \begin{equation*}
     \begin{cases}
      \  \partial\varphi(U) = (-\Delta u_1 + u_1, -\Delta u_2 + u_2 ),   
\\[2mm]
\  D(\partial\varphi) = \{ U=(u_1,u_2) ~\!;~\! u_i \in H^2(\Omega) \  
                                      - \partial_\nu u_i(x) \in \gamma_i(u_i(x)) 
                                          \ (i=1,2)   
                                           \quad\mbox{a.e. on }\partial\Omega\}.
     \end{cases}
\end{equation*}
   Furthermore the elliptic estimate \eqref{elliptic} with $u$ replaced by $u_i$ 
    (\(i=1,2\)) holds true for all $U \in D(\partial \varphi)$.

   Then by putting $ B(U) := ( -u_1 ~\! u_2 + (b-1) ~\! u_1, - u_2 - a ~\! u_1)$, 
    (NR)$^\gamma$ can be reduced to 
      the following abstract evolution equation in $H$.
\begin{equation*}
{\rm (CP)}^\gamma \ 
\begin{cases}
\ \displaystyle\frac{d}{dt} U(t) + \partial \varphi(U(t)) + B(U(t)) \ni 0,  
\quad t>0, 
\\[2mm]
\ U(0) = U_0 =(u_{10}, u_{20}).  
\end{cases}	      
\end{equation*}
  In order to apply ``$L^\infty$-Energy Method'', we again introduce the following 
    cut-off functions $I_{K_{i,M}}(\cdot)$ (\(i=1,2\)):
\begin{equation*}
    I_{K_{i,M}}(U) := 
        \begin{cases}
         \ 0  & U \in K_{i,M} := \{ ~\! U=(u_1,u_2) \in H ~\! ; ~\! |u_i(x)| \leq M 
          \quad a.e. \ x \in \Omega ~\! \}, 
\\[2mm]
           \ + \infty & U \in H \setminus K_{i,M},
        \end{cases} 
\end{equation*}      
   and put 
\begin{equation*}
   \varphi_M(U) := \varphi(U) + I_{K_{1,M}}(U) + I_{K_{2,M}}(U).       
\end{equation*}
  Then we get 
\begin{equation*}
     \partial \varphi(U) = \partial \varphi(U) + \partial I_{1,M}(U) + \partial I_{2,M}(U) 
        \quad \forall U \in D(\partial \varphi) \cap K_{1,M} \cap K_{2,M}.
\end{equation*}  
  Consider the following auxiliary equation:
\begin{equation*}
   {\rm (CP)}_M^\gamma \ 
     \begin{cases}
      \ \displaystyle\frac{d}{dt} U(t) + \partial \varphi_M(U(t)) + B(U(t)) \ni 0,  
       \quad t>0, 
 \\[2mm]
        \ U(0) = U_0,  
 \end{cases}
 \end{equation*} 
 where we choose $M>0$ such that 
 \begin{equation*}
     M = \|U_0\|_{L^\infty} + 2 
          := \|u_{10}\|_{L^\infty} + \|u_{20}\|_{L^\infty} + 2.
 \end{equation*}
  Then as in the proof of Proposition \ref{LWP}, we can easily show that 
   (CP)$_M^\gamma$, which is equivalent to the following (NR)$_M^\gamma$, 
     admits a unique global solution $U(t)=( u_1(t), u_2(t) )$.
\begin{equation*}%\label{DNR}
{\rm (NR)}^\gamma_M \
\left\{
\begin{aligned}
& \ \partial_t u_1 - \Delta u_1  + \beta_M(u_1) = u_1 u_2 - b u_1, 
&& t>0,~ x \in \Omega,
\\
& \ \partial_t u_2-\Delta u_2 + \beta_M(u_2) = a u_1,  
&& t>0,~x\in\Omega,
\\
& \  \partial_{\nu}u_1 + \gamma_1(u_1)  
= \partial_{\nu}u_2 + \gamma_2 (u_2) = 0,
&& t>0,~x \in \partial\Omega,
\\
& \ u_1(0,x) = u_{10}(x) \ge 0,~u_2(0,x) = u_{20}(x) \ge 0,
&& \phantom{t>0,~}  x \in \Omega.
\end{aligned}
\right.
\end{equation*}               
        
    Then in parallel with \eqref{est:u:Linfty}, multiplying the first 
      and second equations of (NR)$_M^\gamma$ by $|u_1|^{r-2}u_1$ and $|u_2|^{r-2}u_2$, 
        we can obtain 
\begin{equation*}
\|U(t)\|_{L^\infty} \leq \|U_0\|_{L^\infty} + \int_0^t \ell(\|U(s)\|_{L^\infty}) ds 
     \quad \text{with} \ \ell(r) = a r + r^2, 
\end{equation*}        
   where $\|U\|_{L^\infty} = \|(u_1,u_2)\|_{L^\infty} := \|u_1\|_{L^\infty} + \|u_2\|_{L^\infty}$. 
       Then we can repeat the same arguments as those in the proof of 
         Proposition \ref{LWP}.  
          Furthermore multiplying the first and second equations of (NR)$^D$ by 
            $u_1^- := \max (-u_1,0)$ and $u_2^- := \max (-u_2,0)$, we can easily deduce 
\begin{align*}
   \frac12 \frac{d}{dt} ( \|u_1^-(t)\|_{L^2}^2 + \|u_2^-(t)\|_{L^2}^2 ) 
      & \leq \|u_2\|_{L^\infty} \|u_1^-(t)\|_{L^2}^2
                 + a ~\! \|u_1^-(t)\|_{L^2} \|u_2^-(t)\|_{L^2} 
\\[2mm]
      & \leq ( \|u_2\|_{L^\infty} + a ) ~\! ( \|u_1^-(t)\|_{L^2}^2
                  + \|u_2^-(t)\|_{L^2}^2 ).
\end{align*}  
  Then by Gronwall's inequality, we get $u_1^-(t) = u_2^-(t)=0$ for all $t$, 
                i.e., $(u_1,u_2)$ is a non-negative solution (see \cite{KO}). 
 ( The non-negativity of solutions can be also derived from application of  
   Theorem \ref{com} for (NR)$^\gamma$ with the coupling term $u_1 ~\! u_2 $ 
     replaced by $ u_1^+ ~\! u_2$. ) 
                
    Here we prepare the following lemma concerning the existence of blow-up 
     solutions of (NR)$^D$.
%%%%%%%%%%%%%%%%%%%%%%%%%%%%%%%%%%%%%%%%%%%%%%%%%%%%%%%%%%%%%%%%%%%%%%%%%%%%%%%
%%%%%%%%    Proposition 3.6       %%%%%%%%%%%%%%%%%%%%%%%%%%%%%%%%%%%%%%%%%%%%%
%%%%%%%%%%%%%%%%%%%%%%%%%%%%%%%%%%%%%%%%%%%%%%%%%%%%%%%%%%%%%%%%%%%%%%%%%%%%%%%
\begin{proposition}\it\label{prop3.6}
	Assume that \( (u_{10}, u_{20})\) belongs to \( \mathbb{L}^\infty_+ (\Omega) \) 
	  and satisfies  
\begin{equation}\label{cond:initial:bu}
  \int_\Omega ( a ~\!u_{10}(x) + b~\! u_{20}(x) - \frac12 u_{20}^2(x) ) ~\! \phi_1(x)~\! dx \geq 0, 
    \quad \int_\Omega u_{20}(x) ~\! \phi_1(x) ~\!dx > 2 (b + \lambda_1).
\end{equation}
	Then the solution $U(t)=(u_1(t), u_2(t))$ of {\rm (NR)}$^D$ blows up in finite time. 
	 Here $\lambda_1$ and $\phi_1$ are the first eigenvalue and its associate 
	   normalized positive eigenfunction of \eqref{eigen}. 
\end{proposition}
%%%%%%%%%%%%%%%%%%%%%%%%%%%%%%%%%%%%%%%%%%%%%%%%%%%%%%%%%%%%%%%%%%%%%%%%%%%%%%%%
%Japanese Comments(see tex file)
% ???????????? C_0 ? z(0)??????????? u_{20} ???????????
%???? C^\ast ? u_{20} ?????????????????? (3.8) ???? u_{20} ???
%???? u_{20} ???????????????? u_{20} ??????????
%%%%%%%%%%%%%%%%%%%%%%%%%%%%%%%%%%%%%%%%%%%%%%%%%%%%%%%%%%%%%%%%%%%%%%%%%%%%%%%%
\begin{proof} 
	Suppose that $U(t)$ is a global solution. 
	 Then multiplying the first and second equations of (NR)$^D$ by \(\phi_1\), 
	   we obtain
	\begin{align}
 	  & \frac{d}{dt}\left( \int_{\Omega} u_1\varphi_1dx \right) 
 	     + (b+\lambda_1)\left( \int_{\Omega} u_1 \phi_1dx \right) 
 	       = \int_{\Omega} u_1 u_2 \phi_1 dx, 
 	           \label{eq:u1phi} 
\\[2mm]
	  & \frac{d}{dt}\left( \int_{\Omega} u_2 \phi_1 dx \right) 
	     + \lambda_1\int_{\Omega}u_2 \phi_1 dx 
	       = a \int_{\Omega} u_1 \phi_1 dx.
              \label{eq:u2phi}
	\end{align}
	Following \cite{Q1}, we set
	\begin{equation*}
	   y(t) := \int_{\Omega} u_2(t) \phi_1dx, \quad 
	     z(t) := y'(t) + (b+\lambda_1) y(t) -\frac{1}{2}\int_{\Omega} u_2^2(t) \phi_1 dx.
	\end{equation*}
  Then by \eqref{eq:u2phi} and \eqref{eq:u1phi}, we get 
\begin{align}
    y''(t) & = - \lambda_1 y'(t) + a \int_\Omega u_1'(t) \phi_1 dx  \notag
 \\
           & = - \lambda_1 y'(t) 
                 - (b + \lambda_1) \int_\Omega a u_1 \phi_1 dx 
                   + \int_\Omega a u_1 u_2 \phi_1 dx.   \label{eq:y2}
\end{align}
  We substitute $a u_1 = \partial_t u_2 - \Delta u_2$ in \eqref{eq:y2}, 
   then by integration by parts we have   
	\begin{equation*}
	  y''(t) + ( b +
	   2 \lambda_1 )y'(t) + \lambda_1(b+\lambda_1) y(t) 
	     =  \frac{1}{2}\frac{d}{dt}\left( \int_{\Omega} u_2^2 \phi_1dx \right) 
	          + \int_{\Omega} |\nabla u_2|^2 \phi_1dx  
	             +\frac{\lambda_1}{2}\!\!\int_{\Omega} u_2^2 \phi_1dx,
	\end{equation*}
  whence follows 
	\begin{equation*}
	z'(t) \ge -\lambda_1 z(t).
	\end{equation*}	
 Therefore we get \(z(t)\ge z(s)e^{-\lambda_1 (t-s)} \) for $0<s<t$.  
   Here \eqref{eq:u2phi} and \eqref{cond:initial:bu} yield  
	\begin{align*}
	z(s)  & =  y'(s) + (b+\lambda_1) ~\! y(s) - \frac12 \int_\Omega u_2^2(s) \phi_1 dx
\\
           & = \int_\Omega ( a ~\!u_1(s) + b ~\! u_2(s) - \frac12 u_2^2(s)  ) ~\! \phi_1 dx
\\  
           & \to \int_\Omega ( a ~\! u_{10} + b ~\! u_{20} - \frac12 u_{20}^2  ) 
              ~\! \phi_1 dx \geq 0  
              \quad \text{as} \ s \to 0, 
	\end{align*}
   since $u_1(t), u_2(t) \in C([0,1];L^2(\Omega)) \cap L^\infty(0,1;L^\infty(\Omega))$. 
    Hence we see that $z(t) \geq 0$ for all $t>0$, i.e., we have 
\begin{align}
   y'(t) & \geq - (b + \lambda_1 ) ~\! y(t) + \frac12 \int_\Omega u_2^2(t) \phi_1 dx
\notag
\\
         & \geq  - (b + \lambda_1 ) ~\! y(t) + \frac12 ~\! y^2(t)
\notag
\\ 
         & \geq \frac12 ~\! y(t) ( y(t) - 2 (b + \lambda_1) ).
 \label{est:below:yprime}
\end{align}
  Then \eqref{est:below:yprime} assures that $y(t)$ blows up in finite time 
   if $y(0) > 2 ( b + \lambda_1)$.
\end{proof}
%%%%%%%%%%%%%%%%%%%%%%%%%%%%%%%%%%%%%%%%%%%%%%%%%%%%%%%%%%%%%%%%%%%%%%%%%%%%%%
  In order to make it clear that solutions of parabolic systems 
   differ according to their boundary conditions imposed, we here denote the unique 
     solutions of (NR) and (NR)$^D$ by $U^\gamma(t) = (u^\gamma_1(t), u^\gamma_2(t))$ 
       and $U^D(t) = (u^D_1(t), u^D_2(t))$ with the same initial data 
        $U_0 \in \mathbb{L}^\infty_+ (\Omega)$, respectively.   

  We are going to compare $U^\gamma(t)$ with $U^D(t)$ by applying Theorem \ref{com}. 
    for $U_1=U^D, \ U_2=U^\gamma$. Let       
\begin{align*}
   & m = 2 ; \  a_{i,j}^1 = a_{i,j}^2 = \delta_{i,j} ;   
       \quad a_1^1=a_2^1=u_{10}, \ a_1^2=a_2^2=u_{20} ;
         \quad \beta_1^1=\beta_2^1=\beta_1^2=\beta_2^2=0 ; 
\\[2mm] 
   & F_1^1(U) = F_2^1(U)= F^1(U) := u_1 u_2 - b u_1, 
     \ F_2^1(U) = F_2^2(U)= F^2(U) := a u_1 ;
\\[2mm] 
   & \gamma_1^1(r) = \gamma_1^2(r)= \gamma^D(r), \quad    
   \gamma_2^i(r) = 
      \begin{cases} 
        \ \alpha_i ~\! |r|^{\gamma_i -2} r  \quad & \ \mbox{for} \ r > 0,
\\[1mm] 
       \ (- \infty,0] \quad & \ \mbox{for} \ r = 0, \phantom{spac}
\\[1mm]
\ \emptyset \quad & \mbox{ for} \ r < 0,          
      \end{cases}
      ( i=1,2),
\end{align*}      
%%%%%%%%%%%%%%%%%%%%%%%%%%%%%%%%%%%%%%%%%%%%%%%%%%%%%%%  
  where $\gamma^D$ is the maximal monotone graph defined by \eqref{BCD}.
   Then (A1), (A2) and (i) of (A4) are obviously satisfied. 
     Moreover as in the proof of Proposition \ref{prop3.3}, 
        we can see that $u_1^D$ and $u_2^D$ obey the homogeneous Dirichlet 
          boundary condition, and that 
           $ - \partial_\nu u_1^\gamma \in \gamma_2^1(u_1^\gamma)$ and 
             $ - \partial_\nu u_2^\gamma \in \gamma_2^2(u_2^\gamma)$ hold, 
              since $u_1^\gamma$ and $u_2^\gamma$ are non-negative solutions. 
     Therefore $ D(\beta_1^1)= D(\beta_1^2)=D(\beta^D) = \{0\}$ and 
       $D(\gamma_2^1)=D(\gamma_2^2) = [0,\infty)$ assure (iii) of (A3). 
       
    Hence to apply Theorem \ref{com}, it suffices to check (ii) of (A4), 
     i.e., $ F^1(U) = u_1 u_2 - b u_1, \ F^2(U)= a u_1$ satisfies (SC). 
       Since $F^1, F^2 \in C^1(\mathbb{R}^2)$, \eqref{cond:F2} is obvious. 
        As for \eqref{cond:F1}, we get 
\begin{equation*}
     \frac{\partial}{\partial u_1} F^2(U) = a >0, 
      \quad 
       \frac{\partial}{\partial u_2} F^1(U) = u_1 \geq 0.
\end{equation*}    
   Consequently, applying Theorem \ref{com}, we conclude 
\begin{align*}
   & T_m(U^\gamma) \leq T_m(U^D) \quad \text{and} 
\\[1mm] 
   & 0 \leq u_1^D(t,x) \leq u_1^\gamma(t,x), \quad 
       0 \leq u_2^D(t,x) \leq u_2^\gamma(t,x) \quad \forall t \in [0,T_m(U^\gamma)) 
          \quad a.e. \ x \in \Omega.
\end{align*}
	Thus by virtue of Proposition \ref{prop3.6}, we have the following corollary.
%%%%%%%%%%%%%%%%%%%%%%%%%%%%%%%%%%%%%%%%%%%%%%%%%%%%%%%%%%%%%%%%%%%%%%%
%%%%%%%%%   Corollary 3.7
%%%%%%%%%%%%%%%%%%%%%%%%%%%%%%%%%%%%%%%%%%%%%%%%%%%%%%%%%%%%%%%%%%%%%%%
\begin{corollary}\label{cor3.7}
   Assume that \( (u_{10}, u_{20})\) belongs to \( \mathbb{L}^\infty_+ (\Omega) \) 
    and satisfies \eqref{cond:initial:bu}.   
     Then the unique solution $U(t)=(u_1(t), u_2(t))$ of {\rm (NR)}
       blows up in finite time. 	
\end{corollary}
%%%%%%%%%%%%%%%%%%%%%%%%%%%%%%%%%%%%%%%%%%%%%%%%%%%%%%%%%%%%%%%%%%%%%%%%
%%%%%%%%    Remark 3.8    %%%%%%%%%%%%%%%%%%%%%%%%%%%%%%%%%%%%%%%%%%%%%%
%%%%%%%%%%%%%%%%%%%%%%%%%%%%%%%%%%%%%%%%%%%%%%%%%%%%%%%%%%%%%%%%%%%%%%%%
\begin{remark}
   The existence of $(u_{10}, u_{20})$ satisfying \eqref{cond:initial:bu} 
     is assured when $ a>0$. For instance, if $u_{10} \geq \frac{1}{2 a} u_{20}^2$ 
       and $u_{20}$ is sufficiently large, then \eqref{cond:initial:bu} is satisfied.
     
     For the case where $a=0$, however, there is no initial 
       data $(u_{10}, u_{20})$ satisfying \eqref{cond:initial:bu}. 
   In fact, $a=0$ implies that $ \sup_{t \geq 0} \|u_2(t)\|_{L^\infty} \leq \|u_{20}\|_{L^\infty}$, 
     then $u_1(t)$ satisfies $\partial_t u_1 - \Delta u_1(t) 
      \leq \|u_{20}\|_{L^\infty} u_1(t) $, whence follows 
        $\|u_1(t)\|_{L^\infty} \leq \|u_{10}\|_{L^\infty} ~\! e^{\|u_{20}\|_{L^\infty} t}$. 
          Consequently every local solution can be continued globally. 
\end{remark}
%%%%%%%%%%%%%%%%%%%%%%%%%%%%%%%%%%%%%%%%%%%%%%%%%%%%%%%%%%%%%%%%%%%%%%%%%
%%%%%%%%%   Remark 3.9       %%%%%%%%%%%%%%%%%%%%%%%%%%%%%%%%%%%%%%%%%%%%
%%%%%%%%%%%%%%%%%%%%%%%%%%%%%%%%%%%%%%%%%%%%%%%%%%%%%%%%%%%%%%%%%%%%%%%%%
\begin{remark}
  The assertion of Corollary \ref{cor3.7} holds true for more general equation 
   (NR)$^\gamma$, provided that $ 0 \in \gamma_i(0) \ (i=1,2)$ is satisfied.
\end{remark}
}
%%%%%%%%%%%%%%%%%%%%%%%%%%%%%%%%%%%%%%%%%%%%%%%%%%%%%%%%%%%%%%%%%%%%%%%%
%%%%%%%%%%%%%%%%%%%%%%%%%%%%%%%%%%%%%%%%%%%%%%%%%%%%%%%%%%%%%%%%%%%%%%%%

%\bibliography{wileyNJD-AMA}

\begin{thebibliography}{30}
	\bibitem{Ba1} V. Barbu, ``\textit{Nonlinear Differential Equations of Monotone Types in Banach Spaces}'', Springer Monographs in Mathematics, 2010.
	\bibitem{Be} P. B\'enilan and J. I. D\'iaz, Comparison of solutions of nonlinear evolution problems with different nonlinear terms, \textit{Israel J. Math.}, \textbf{42}, no. 3 (1982), 241-257.
	\bibitem{B0} H. Br\'ezis, ``\textit{Op\'erateurs Maximaux Monotones et Semigroupes de Contractions dans Espace de Hilbert},'' North Holland, Amsterdam, The Netherlands, 1973.
	\bibitem{B1} H. Br\'{e}zis, Monotonicity methods in Hilbert spaces and some applications to nonlinear partial differential equations, in Contributions to Nonlinear Funct. Analysis, Madison, 1971, (Ed. by E. Zarantonello ), Acad. Press, 1971, p. 101-156. 
	\bibitem{GW1} Y. G. Gu and M. X. Wang, A semilinear parabolic system arising in the nuclear reactors, \textit{Chinese Sci. Bull.}, \textbf{39}, No.19 (1994), 1588-1592.
	\bibitem{KC1} W. E. Kastenberg and P. L. Chambr\'{e}, On the stability of nonlinear space-dependent reactor kinetics, Nucl. Sci. Eng., \textbf{31} (1968), 67-79.
	\bibitem{KO2} K. Kita and M. \^Otani, Bounds for global solutions of a reaction diffusion system with the Robin boundary conditions, \textit{Differ. Equ. Appl.}, \textbf{11}, no.2 (2019), 227-242.
	\bibitem{KO} K. Kita, M. \^Otani and H. Sakamoto, On some parabolic systems arising from a nuclear reactor model with nonlinear boundary conditions, \textit{Adv. Math. Sci. Appl.}, \textbf{27}, No.2 (2018), 193-224.
	\bibitem{LSU} O. A. Lady\v{z}enskaja, V. A. Solonnikov and N. N. Ural'ceva, \textit{Linear and quasilinear equations of parabolic type}, Translations of Mathematical Monographs \textbf{23}, Amer. Math. Soc. 1968.
	\bibitem{O} M. \^{O}tani, Nonmonotone perturbations for nonlinear parabolic equations
	  associated with subdifferential operators, Cauchy Problems, 
      {\it J. Differential Equations} {\bf 46} (1982), no. 2, 268-299.	
	\bibitem{O1} M. \^{O}tani, \(L^{\infty}\)-energy method, basic tools and usage, \textit{Differential Equations, Chaos and Variational Problems,} Progress in Nonlinear Differential Equations and Their Applications, \textbf{75}, Ed. by Vasile Staicu, Birkhauser (2007), 357-376.
	\bibitem{PS3} L. E. Payne and P. W. Schaefer, Blow-up in parabolic problems under Robin boundary conditions, \textit{Applicable Analysis}, \textbf{87}, No. 6 (2008), 699-707.
	\bibitem{Q1} P. Quittner, Transition from decay to blow-up in a parabolic system, Equadiff 9 (Brno, 1997). Arch. Math. (Brno) 34 (1998), no. 1, 199-206.
	\bibitem{QS} P. Quittner and P. Souplet, ``\textit{Superlinear Parabolic Problems, Blow-up, Global Existence and Steady States, Second edition},'' Birkh\"auser Basel, 2019.
	\bibitem{RT} A. Rodr\'{\i}guez-Bernal and A. Tajdine, Nonlinear balance for reaction-diffusion equations under nonlinear boundary conditions: dissipativity and blow-up, \textit{J. Differential Equations}, \textbf{169}, no. 2 (2001), 332-372.
\end{thebibliography}

\end{document}